\pdfoutput=1
\newcommand{\showdate}{false}

\documentclass[reqno,a4paper]{amsart}
\usepackage{microtype}
\usepackage{color}
\usepackage{amssymb}
\usepackage[utf8]{inputenc}
\usepackage{a4wide}
\usepackage{paralist}
\usepackage{enumitem}
\usepackage[all]{xy}
\usepackage{ifthen}
\usepackage{slashed}
\RequirePackage{xspace}

\usepackage[all]{xy}

\DeclareFontFamily{U} {MnSymbolC}{}
\DeclareFontShape{U}{MnSymbolC}{m}{n}{
  <-6> MnSymbolC5
  <6-7> MnSymbolC6
  <7-8> MnSymbolC7
  <8-9> MnSymbolC8
  <9-10> MnSymbolC9
  <10-12> MnSymbolC10
  <12-> MnSymbolC12}{}
\DeclareFontShape{U}{MnSymbolC}{b}{n}{
  <-6> MnSymbolC-Bold5
  <6-7> MnSymbolC-Bold6
  <7-8> MnSymbolC-Bold7
  <8-9> MnSymbolC-Bold8
  <9-10> MnSymbolC-Bold9
  <10-12> MnSymbolC-Bold10
  <12-> MnSymbolC-Bold12}{}
\DeclareSymbolFont{MnSyC} {U} {MnSymbolC}{m}{n}
\DeclareMathSymbol{\largepentagram}{\mathord}{MnSyC}{133}

\newcommand{\calz}{\mathcal{Z}}
\newcommand{\cala}{\mathcal{A}}
\newcommand{\calb}{\mathcal{B}}

\newcommand{\calg}{\mathcal{G}}

\newcommand{\calk}{\mathcal{K}}
\newcommand{\calt}{\mathcal{T}}
\newcommand{\calc}{\mathcal{C}}
\providecommand{\texorpdfstring}[2]{#1}

\newcommand{\tkc}{$(n{-}1)$-connected}
\newcommand{\bmp}{Bianchi-Massey tensor}
\newcommand{\nmt}{\mathcal{P}}
\newcommand{\nmtb}{\overline{\nmt}}
\newcommand{\cmap}{\largepentagram}
\DeclareMathAlphabet{\matheur}{U}{eur}{m}{n}
\newcommand{\rmap}{\matheur{m}}
\newcommand{\rfunc}{\matheur{R}}
\newcommand{\ndom}{\mathcal{D}}
\newcommand{\tripsp}{\kersym{H^* \otimes \kc^*}}
\newcommand{\trip}{\calt}

\newcommand{\wh}{\widehat}

\DeclareMathOperator{\im}{Im}

\DeclareMathOperator{\sgn}{sign}
\newcommand{\ie}{\emph{i.e.} }

\newcommand{\eg}{\emph{e.g.} }

\newcommand{\princ}{\mathcal{F}}

\newcommand{\pol}{P}

\newcommand{\sym}{\textup{Sym}}

\newcommand{\gsym}{\calg}
\newcommand{\agsym}{\slashed\calg}

\newcommand{\subsp}{L}
\newcommand{\kersym}[1]{K[#1]}

\newcommand{\bmsp}{\calb}

\newcommand{\dga}{\cala}
\newcommand{\dgab}{\calb}

\newcommand{\clalg}{\calz}

\newcommand{\kc}{E}

\newcommand{\half}{{\textstyle\frac{1}{2}}}

\newcommand{\third}{{\textstyle\frac{1}{3}}}

\newcommand{\quart}{\textstyle\frac{1}{4}}

\newcommand{\Z}{\mathbb{Z}}
\newcommand{\Q}{\mathbb{Q}}

\newcommand{\idf}{\matheur{I\mkern-2mu d}}
\newcommand{\Id}{\textup{Id}}

\newcommand{\Ker}{\mathop{\mathrm{Ker}}\nolimits}
\newcommand{\Coker}{\mathop{\mathrm{Coker}}\nolimits}

\newcommand{\Hom}{\mathop{\mathrm{Hom}}\nolimits}

\newcommand{\wt}[1]{\widetilde #1}

\newcommand{\gen}[1]{\langle#1\rangle}

\newtheorem{thm}{Theorem}[section]
\newtheorem{prop}[thm]{Proposition}
\newtheorem{lem}[thm]{Lemma}
\newtheorem{conj}[thm]{Conjecture}
\newtheorem{cor}[thm]{Corollary}

\theoremstyle{definition}
\newtheorem{defn}[thm]{Definition}

\theoremstyle{remark}
\newtheorem{rmk}[thm]{Remark}
\newtheorem*{rmk*}{Remark}
\newtheorem{ex}[thm]{Example}

\setlist{leftmargin=*}
\renewcommand{\hom}{\mathrm{Hom}}

\title{Rational homotopy and simply-connected 8-manifolds}
\author[Cs. Nagy]{Csaba Nagy}
\author[J. Nordström]{Johannes Nordström}
\address{Department of Mathematical Sciences,
University of Bath,
Bath BA2 7AY, UK}
\email{cn510@bath.ac.uk}
\email{j.nordstrom@bath.ac.uk}

\begin{document}

\begin{abstract}
We introduce a rational homotopy invariant $\nmt$ of a topological space, which
is a quintic tensor on the cohomology. For $n \geq 2$, formality of a
closed \tkc{} manifold of dimension up to $5n-2$ is equivalent to vanishing
of $\nmt$ and the Bianchi-Massey tensor introduced by Crowley and the second
author. We show also that elements of the group $\Theta_8(\bigvee^r S^2)$ of closed
simply-connected spin 8-manifolds with the cohomology of an $r$-fold connected sum
of $S^2 \times S^6$ are determined up to torsion by the value of $\nmt$.
\end{abstract}

\ifthenelse{\boolean{\showdate}}{\date{\today}}{}
\maketitle

\ifthenelse{\boolean{\showdate}}{\vspace{-0.8\baselineskip}}{}
\vspace{-\baselineskip}

\section{Introduction}

We introduce and study a new rational homotopy invariant $\nmt$ of a
topological space $X$, which we call the \emph{pentagonal Massey tensor}.
It is a linear map of degree $-2$ to $H^*(X)$ from a
subspace of the fifth tensor power of $H^*(X)$, and captures information
similar to 4-fold Massey products. It is similar in style to the \bmp{}
introduced by Crowley and the second author \cite{bmp}
(see \S\ref{subsec:triples}). The definition of $\nmt$ will be given in
Sections \ref{s:8comp} and \ref{sec:general}, and we summarise our main results below.

Throughout the paper all homology and cohomology will be with rational coefficients unless otherwise indicated.

\subsection{8-manifolds}

Our main motivation for defining the pentagonal Massey tensor $\nmt$
is to study simply-connected 8-manifolds.
The simplest closed simply-connected manifolds where $\nmt$ can be non-trivial
are $8$-manifolds with the homology of a connected sum of some copies of
$S^2 \times S^6$.
For such a manifold $M$ the pentagonal Massey tensor is a well-defined linear
map $\rfunc(H^2(M)) \rightarrow H^8(M)$
(described in \S\ref{subsec:special_def}), where $\rfunc$ is defined as follows. 

\begin{defn}
\label{def:ndom}
For a vector space $V$ let $\pol^k V$ and $\Lambda^k V$ denote the degree-$k$ components of the polynomial and exterior algebra respectively, \ie the quotients of $V^{\otimes k}$ by the relation of symmetry or antisymmetry. We define $\rfunc(V)$ to be the kernel of 
\begin{equation}
\label{eq:rmap}
\begin{aligned}
\rmap: V \otimes \Lambda^2 \pol^2 V &\to \pol^2 V \otimes \pol^3 V, \\
q \otimes (xy \wedge zw) &\mapsto xy \otimes zwq - zw \otimes xyq
\end{aligned}
\end{equation}
\end{defn}

Note that $\rfunc$ is a functor. By Lemma \ref{lem:ndom}, if $\dim V = r$,
then $\dim \rfunc(V)= 6 \binom{r+2}{5}$.

\medskip
We will denote by $\Theta_8(\bigvee^r S^2)$ the set of diffeomorphism classes of spin (polarized) $8$-manifolds with the homology of $\#^r S^2 \times S^6$ (see
Definition \ref{def:theta}). It is a group under ``connected sum along the
$2$-skeleton", and it was computed by the first author \cite{nagy20} to be
\[ 
\textstyle
\Theta_8(\bigvee^r S^2) \cong \Z^a \oplus (\Z/2)^b \text{\,,\quad where } a = 6{r+2 \choose 5} \text{ and } b = 2{r+3 \choose 4} - {r-1 \choose 2} + 2 \text{\,.} 
\]

For a closed oriented 8-manifold $M$ the fundamental class determines an identification
${H^8(M) \cong \Q}$, so $\nmt$ amounts to a linear functional
\[
\nmtb(M) \in \rfunc(H^2(M))^{\vee} = \Hom(\rfunc(H^2(M)), \Q) \text{\,.}
\]
Further, a manifold $M$ in $\Theta_8(\bigvee^r S^2)$ by definition comes with an
identification $H_2(M;\Z) \cong \Z^r$ (so $H_2(M) \cong \Q^r$ and
$H^2(M) \cong (\Q^r)^{\vee}$), hence in this setting $\nmtb(M)$ can be regarded
as an element of $\rfunc((\Q^r)^{\vee})^{\vee} \cong \rfunc(\Q^r)$.
The first aim of this paper
is to compute $\nmt$ (or, equivalently, $\nmtb$) for elements of
$\Theta_8(\bigvee^r S^2)$. We find that $\nmt$ captures all the non-torsion information,
so in particular it determines elements of $\Theta_8(\bigvee^r S^2)$ up to finite
ambiguity. More precisely, we prove in \S\ref{subsec:nmt_theta} that

\begin{thm}
\label{thm:main2}
$\nmtb$ defines an isomorphism $\Theta_8(\bigvee^r S^2) \otimes \Q \cong \rfunc(\Q^r)$.
\end{thm}

\subsection{Relation to fourfold Massey products}

Much of the theory still works nicely even if we go from manifolds in $\Theta_8(\bigvee^r S^2)$ to
allowing $H_4$ to be non-trivial.

\begin{defn} \label{def:eman}
A closed manifold $M$ is called an \textit{E-manifold} if it is simply-connected and its homology is concentrated in even dimensions, that is, $H_{2k+1}(M;\Z) \cong 0$ for every $k$.
\end{defn}

\begin{ex}
Examples of E-manifolds include homotopy spheres (in even dimensions), complex and quaternionic projective spaces, Bott-manifolds and complex projective complete intersections (in even complex dimensions). 
\end{ex}

In particular, $\nmt$ is still uniquely defined for an $8$-dimensional
E-manifold, although the domain $\rfunc(H^2)$ is replaced by
$\ndom := (H^2 \otimes \Lambda^2 \kc) \cap \rfunc(H^2)$, where $\kc$ is the kernel of the cup product map $\pol^2 H^2 \to H^4$.
Moreover, the ambiguities in the definition of fourfold Massey products also
remain more manageable than in general.

\begin{defn} \label{def:penta}
For a vector space $V$ we define the multilinear map $\cmap : V^{\times 5} \rightarrow V \otimes \Lambda^2 \pol^2 V$ by 
\[
\cmap(x_1, x_2, x_3, x_4, x_5) = \sum_{\text{cyc}} x_1 \otimes (x_2x_3 \wedge x_4x_5) \text{\,.}
\]
Note that $\cmap(x_1, x_2, x_3, x_4, x_5) \in \rfunc(V)$. 
\end{defn}

Note also that if $M$ is an $8$-dimensional E-manifold and the elements $x_1, x_2, x_3, x_4, x_5 \in H^2(M)$ satisfy $x_1x_2 = x_2x_3 = x_3x_4 = x_4x_5 = x_5x_1 = 0 \in H^4(M)$, then $\cmap(x_1, x_2, x_3, x_4, x_5) \in \ndom$. In \S\ref{subsec:fourfold} we prove the following: 

\begin{prop}
For any $8$-dimensional E-manifold $M$ the fourfold Massey products are
completely determined by $\nmt$ via the formula 
\[ 
\gen{x_1,x_2,x_3,x_4}x_5 = \nmt\left(\cmap(x_1, x_2, x_3, x_4, x_5)\right) \in H^8(M) \text{\,.}
\]
If $H_4(M) = 0$ then also $\nmt$ is determined by the fourfold Massey products.
\end{prop}

We can thus also say that elements of $\Theta_8(\bigvee^r S^2)$ are determined up to
$2^{2{r+3 \choose 4} - {r-1 \choose 2} + 2}$ possibilities by their fourfold
Massey products.

\subsection{The role of \texorpdfstring{$\nmt$}{P} in rational homotopy classification}

Now we move on to considering $\nmt$ for more general closed manifolds. The general version of $\nmt$ is defined in \S\ref{ss:gen-def}. In \cite{bmp} it was shown that for $n \geq 2$, the rational homotopy type of a
closed \tkc{} manifold of dimension up to $5n-3$ is determined by its
cohomology algebra and \bmp{} $\princ$. In particular, such a manifold is
formal if and only if $\princ = 0$. For an \tkc{} manifold of
dimension $\geq 5n-2$, fourfold Massey products and the pentagonal Massey
tensor can be non-trivial. In the borderline case, we show
in \S\ref{subsec:formality} that $\nmt$ can be used to characterise formality. 

\begin{thm}
\label{thm:main1}
For $n\geq 2$, a closed \tkc{} manifold of dimension $5n-2$ is formal if and only if $\princ = 0$ and $\nmt = 0$.
\end{thm}

While the \bmp{} is always uniquely defined, $\nmt$ has a slightly complicated
dependence on cochain choices in general (described
in Theorem \ref{thm:choices}). However, if $\princ = 0$ then $\nmt$
is uniquely defined too (Remark \ref{rmk:bmt_zero}), so the hypothesis in
Theorem \ref{thm:main1} is unambiguous.

\begin{rmk}
Combining Theorems \ref{thm:main2} and \ref{thm:main1} gives that the subset of
formal elements of $\Theta_8(\bigvee^r S^2)$ is precisely the torsion subgroup.
\end{rmk}

The term ``formality'' derives from the rational homotopy type of a formal
space being a ``formal consequence'' of the cohomology algebra.
In particular, Theorem \ref{thm:main1} means that if $\princ$ and $\nmt$
of a \tkc{} $(5n-2)$-manifold vanish then its rational homotopy type is
determined by the cohomology algebra.
We expect that if we weaken the hypotheses of Theorem~\ref{thm:main1} to allow
non-zero~$\nmt$, then the rational homotopy type is determined by the
cohomology algebra together with $\nmt$. 

\begin{conj}
\label{conj:special}
Let $F : H^*(X) \to H^*(Y)$ be an isomorphism of the cohomology algebras
of closed \tkc{} manifolds of dimension $\leq 5n-2$ with trivial \bmp. 
Then $F$ is realised by a rational homotopy equivalence if and only if
$F$ intertwines the pentagonal Massey tensors.
\end{conj}

If the \bmp{} is allowed to be non-zero then it is a little more intricate
to formulate the correct statement because of the ambiguities in $\nmt$,
see Conjecture \ref{conj:type}.
Since a closed \tkc{} manifold of dimension $\leq 6n-4$ cannot have non-trivial
Massey products of order greater than 4 we expect that the situation is largely
the same in that dimension range, but it may be necessary to add a slight
generalisation of the \bmp{} as discussed in Remark~\ref{rmk:gen_bmp}.

\subsection*{Acknowledgements}

The authors thank Manuel Amann, Diarmuid Crowley and Alastair King for useful
discussions, and the Simons foundation for its support of their
research under the Simons Collaboration on ``Special Holonomy in Geometry,
Analysis and Physics'' (grant \#488631, Johannes Nordstr\"om).

\section{The degree 8 component of \texorpdfstring{$\nmt$}{P}} \label{s:8comp}

We describe the pentagonal Massey tensor $\nmt$ in general
in \S\ref{sec:general}. Let us for now describe the component relevant for
closed simply-connected $8$-manifolds. In this section we will work solely in
the algebraic setting.

\subsection{The representation \texorpdfstring{$\rfunc(V)$}{R(V)}}

Let us first make some brief remarks about the domain of $\nmt$.
Given a vector space $V$, we defined $\rfunc(V) = \ker \rmap$ in Definition
\ref{def:ndom}, which can be thought of as a representation of $GL(V)$.
Recall that any irreducible representation of $GL(V)$ is a Weyl module
$S(\lambda)$ indexed
by a partition $\lambda$ of an integer $n$ (for which the representation embeds
in $V^{\otimes n}$).

\begin{lem}
\label{lem:ndom}
\begin{enumerate}
\item If $V$ is a rational vector space of dimension $r$, then
$\dim \rfunc(V) = 6{r+2 \choose 5}$.
\item $\rfunc(V)$ is the irreducible representation $S(3,1,1)$ of $GL(V)$
(\ie the Weyl module corresponding to the partition $5 = 3+1+1$).
\end{enumerate}
\end{lem}

\begin{proof}
Both claims follow from the observation that
the image of $\rmap$ is precisely the kernel of the symmetrisation map
$\pol^3 V \otimes \pol^2 V \to \pol^5 V$, which is easy to establish by
picking a basis for $V$: modulo the image of $\rmap$, a monomial $(qxy)(zw)$
in basis elements is equal to $(qzw)(xy)$, and it is clear that two monomials
are related by a sequence of such swaps if and only if they involve the same 5
basis elements, \ie the images in $P^5 V$ are the same monomial.
\begin{enumerate}
\item The exactness of
$0 \to \rfunc(V) \to V \otimes \Lambda^2 \pol^2 V \to \pol^3 V \otimes \pol^2 V
\to \pol^5 V \to 0$ allows us to evaluate
\[ \dim \rfunc(V) = r {{r+1 \choose 2} \choose 2} - {r+2 \choose 3}{r+1 \choose 2}+ {r+4 \choose 5} . \]
\item The other terms in the exact sequence are easy to decompose into irreducibles, see Fulton and Harris \cite[\S6]{fulton:rep}.
$\pol^5 V = S(5)$, and by Pieri's formula
$\pol^3 V \otimes \pol^2 V = S(5) \oplus S(4,1) \oplus S(3,2)$, so
$\im \rmap$ must be $S(4,1) \oplus S(3,2)$.
On the other hand Pieri's formula also gives
$\Lambda^2 \pol^2 V = S(3,1)$ (see \cite[Exercise 6.16]{fulton:rep}) and hence
$V \otimes \Lambda^2 \pol^2 V = S(4,1) \oplus S(3,2) \oplus S(3,1,1)$, so
$\rfunc(V)$ is $S(3,1,1)$.
\end{enumerate}
Having established (ii), one could of course alternatively deduce (i) by
applying a standard formula \cite[Theorem 6.3]{fulton:rep} for the dimension
of a Weyl module. 
\end{proof}

\subsection{Definition of the degree 8 component of \texorpdfstring{$\nmt$}{P}}
\label{subsec:special_def}

Let $\dga$ be a DGA over $\Q$ (\eg $\Omega^*_{PL}(X)$ for a topological space $X$),
with cohomology $H^*$. Let $\kc^4 = \kc_{\dga} \subseteq \pol^2 H^2$ be the kernel
of the product map $\pol^2 H^2 \to H^4$, and let
$\ndom = \ndom_{\dga} \subseteq H^2 \otimes \Lambda^2 \kc$ be the kernel of the restriction of
\eqref{eq:rmap};
equivalently, $\ndom_{\dga} = (H^2 \otimes \Lambda^2 \kc) \cap \rfunc(H^2)$.
\begin{rmk}
In terms of Definition \ref{def:ndom_gen} in the more general setting,
$\ndom_\dga$ coincides with the degree $10$ component of the graded vector
space $\ndom^*(H^*(\dga))$ associated to the graded algebra $H^*(\dga)$.
\end{rmk}

Let $\clalg \subseteq \dga$ be the subalgebra of closed elements. Pick a right
inverse $\alpha : H^2 \to \clalg^2$ for the projection to cohomology.
Then the restriction of $\alpha^2 : \pol^2 H^2 \to \clalg^4$ to $\kc$ takes
exact values, so one may pick $\gamma : \kc \to \dga^3$ such that $\alpha^2_{|\kc} = d\gamma$.

\begin{prop} \label{prop:ag2-closed}
The restriction of
\begin{equation}
\label{eq:chains}
\alpha \gamma^2 : H^2 \otimes \Lambda^2 \kc \to \dga^8
\end{equation}
to $\ndom$ takes closed values.
\end{prop}

We will clarify the notation and prove Proposition \ref{prop:ag2-closed} in \S\ref{subsec:special_proof}.

\begin{defn}
The pentagonal Massey tensor is the linear map
\[ \nmt : \ndom \to H^8 \]
induced by the restriction of $\alpha \gamma^2$ to $\ndom$.
\end{defn}

In general $\nmt$ has a slightly subtle dependence on the choices of $\alpha$
and $\gamma$, as set out in Theorem \ref{thm:choices}. We find it useful
to prove separately in the next subsection the following special case.

\begin{thm}
\label{thm:special}
Let $\dga$ be a DGA with $H^3 = 0$. Then \eqref{eq:chains} induces a
well-defined linear map $\nmt : \ndom \to H^8$, independent of the
choices of $\alpha$ and $\gamma$.
\end{thm}

\begin{ex}
Suppose $X$ is a simply-connected 8-manifold with $H^3(X) = 0$, and
$a, x_1, x_2, x_3 \in H^2(X)$ such that $ax_i = 0 \in H^4(X)$. Then
$x_1 \otimes (ax_2 \wedge ax_3) + x_2 \otimes (ax_3 \wedge ax_1) + x_3 \otimes (ax_1 \wedge ax_2) \in \ndom$. If $\alpha, \beta_i \in \Omega^2(X)$ represent $a$ and $x_i$ respectively
and $\gamma_i \in \Omega^3(X)$ satisfy $d\gamma_i = \alpha\beta_i$, then
the result of evaluating $\nmt$ on
$x_1 (ax_2 \wedge ax_3) + x_2 (ax_3 \wedge ax_1) + x_3 (ax_1 \wedge ax_2)$ is the cohomology class
\begin{equation}
\label{eq:fm}
[\beta_1\gamma_2\gamma_3 + \beta_2\gamma_3\gamma_1 + \beta_3\gamma_1\gamma_2] 
\in H^8(X) .
\end{equation}
As a special case of Theorem \ref{thm:special}
(and Remark \ref{rmk:obstruction}) we thus we recover the
claim from  Fern\'andez--Muñoz \cite[Lemma~3.1]{fernandez08} that the class
\eqref{eq:fm} is independent of the choices, and if it is non-zero then
$X$ cannot be formal.
\end{ex}

\begin{ex}
Suppose $H^2 \cong \Q^3(x,y,z)$, $H^4 \cong \Q$ and the product $\pol^2 H^2 \to H^4$ is given by the quadratic form $ax+by+cz \mapsto a^2 + b^2 + c^2$.
An 8-manifold with such a cohomology algebra has the ``hard Lefschetz
property'' in the sense that there is a class $\varphi \in H^4$ such that
multiplication by $\varphi$ defines an isomorphism $H^2 \to H^6$.

By considering $\kc$ as the $SO(3)$-representation $\sym^2_0 \Q^3$
and decomposing $H^2 \otimes \Lambda^2 \kc$ into irreducible $SO(3)$-representations one finds that $\ndom$ is 1-dimensional. One can take
\[
\sum_{\text{cyc}} x \otimes \bigl(yz \wedge (y^2-z^2) - xy \wedge xz \bigr)
\]
as an explicit generator.
One can use \cite[Proposition 3.4(i)]{bmp} to produce minimal DGAs where $\nmt$
is non-trivial on this generator and whose cohomology algebra satisfies
8-dimensional Poincar\'e duality, and Sullivan \cite[Theorem 13.2]{sullivan77}
to realise such a DGA as the minimal model of a closed simply-connected
8-manifold. 
This shows that the result of Cavalcanti \cite[Theorem 4]{cavalcanti06} that
\tkc{} $4n$-manifolds with the hard Lefschetz property and $b_n \leq 2$ are formal cannot be extended to $b_n = 3$.
\end{ex}

\subsection{Well-definedness of the degree 8 component of \texorpdfstring{$\nmt$}{P} when \texorpdfstring{$H^3 = 0$}{H\^3 = 0}}
\label{subsec:special_proof}

We now prove Proposition \ref{prop:ag2-closed} and Theorem~\ref{thm:special}.
Let us first set up some notation to keep the calculations manageable.

For linear maps $\alpha, \beta : H^2 \to \dga^*$, let
$\alpha\beta : \pol^2 H^2 \to \dga^*$ denote the linear map induced by
the symmetric bilinear map
$H^2 \times H^2 \to \dga^*,
(x, y) \mapsto \alpha(x)\beta(y) + \alpha(y)\beta(x)$.
This product is graded symmetric in the sense that if $\alpha$ and $\beta$ take
values in $\dga^r$ and $\dga^s$ respectively, then
$\beta\alpha = (-1)^{rs}\beta\alpha$.
If $\alpha$ takes values in the even degree part of $\dga$, then we can
also define $\alpha^2 : \pol^2 H^2 \to \dga^*$ to be induced by
$(x,y) \mapsto \alpha(x)\alpha(y)$. Note that $\alpha\alpha = 2\alpha^2$.
We will not notationally distinguish the restrictions of these maps
to $\kc \subseteq \pol^2 H^2$.

For $\gamma, \delta : \kc \to \dga^*$ (and $\gamma$ of odd degree $r$),
define $\gamma \wedge \delta: \Lambda^2\kc \to \dga^*$ (and $\gamma^2$) 
analogously; this is graded antisymmetric in the sense that
$\delta \wedge \gamma = (-1)^{1+rs}\gamma \wedge \delta$.

The recurring theme in the proofs below is to show that a function on
$H^2 \otimes \Lambda^2 \kc$ vanishes on $\ndom$ by factoring it through the map
$\rmap : H^2 \otimes \Lambda^2 \kc \to \pol^3H^2 \otimes \kc$
from \eqref{eq:rmap}.
Note that for linear maps $\alpha_1, \alpha_2, \alpha_3 : H^2 \to \dga^*$
and $\gamma : \kc \to \dga^*$, the composition of $\rmap$ with
\begin{align*} \pol^3 H^2 \otimes \kc \to& \dga^*, \\
xyz \otimes p \mapsto&
\alpha_1(x)\alpha_2(y)\alpha_3(z)\gamma(p) +
\alpha_2(x)\alpha_3(y)\alpha_1(z)\gamma(p) +
\alpha_3(x)\alpha_1(y)\alpha_2(z)\gamma(p) \\ +&
\alpha_1(x)\alpha_3(y)\alpha_2(z)\gamma(p) +
\alpha_2(x)\alpha_1(y)\alpha_3(z)\gamma(p) +
\alpha_3(x)\alpha_2(y)\alpha_1(z)\gamma(p)
\end{align*}
equals
\begin{equation}
\label{eq:vanishing}
\alpha_1 (\alpha_2\alpha_3 \wedge \gamma) +
\alpha_2 (\alpha_3\alpha_1 \wedge \gamma) +
\alpha_3 (\alpha_1\alpha_2 \wedge \gamma) :
H^2 \otimes \Lambda^2 \kc \to \dga^* .
\end{equation}
Thus the restriction of any map of the form \eqref{eq:vanishing} to $\ndom$
will vanish.

\begin{proof}[Proof of Proposition \ref{prop:ag2-closed}]
We now verify the claim that the map $\alpha\gamma^2$ takes
closed values on $\ndom$.
As maps $H^2 \otimes \Lambda^2\kc \to \dga^9$,
\[ d(\alpha\gamma^2) = \alpha d\gamma \wedge \gamma = \alpha \alpha^2 \wedge \gamma . \]
The right-hand side is of the form \eqref{eq:vanishing} so that its restriction
to $\ndom$ vanishes.
Thus $\alpha\gamma^2$ maps $\ndom \to \clalg^8$ as claimed. 
\end{proof}

\begin{proof}[Proof of Theorem \ref{thm:special}]
We start by checking that, for any fixed choice of $\alpha$, $\nmt$ is independent of
the choice of $\gamma$. If $\gamma'$ is a different choice, then
$\gamma'-\gamma : \kc \to \dga^3$ takes closed values. The hypothesis
that $H^3 = 0$ forces the difference to be exact, \ie
\[ \gamma' = \gamma + d\eta \]
for some $\eta : \kc \to \dga^2$. Then
\begin{align*}
&\alpha(\gamma')^2 - \alpha\gamma^2 +d (\alpha(\gamma \wedge \eta - \half \eta \wedge d\eta)) \\
= \; &\alpha \gamma \wedge d\eta + \alpha (d\eta)^2 + \alpha (d(\gamma \wedge \eta) - (d\eta)^2) \\
= \; & \alpha d\gamma \wedge \eta = \alpha \alpha^2\wedge \eta .
\end{align*}
The right-hand side is again of the form \eqref{eq:vanishing} so that its
restriction to $\ndom$ vanishes. Hence the restriction of
$\alpha(\gamma')^2 - \alpha\gamma^2$ to $\ndom$ takes exact values as required.

It remains to show that given $\alpha, \gamma$, and another choice
$\alpha'$, there is also a choice of $\gamma'$ so that the restrictions to
$\ndom$ of $\alpha\gamma^2$ and $\alpha'(\gamma')^2$ differ by exact values.
We can write $\alpha' = \alpha + d\beta$ for some $\beta : H^2 \to \dga^1$,
and then take $\gamma' = \gamma + \sigma_{|\kc}$ for
$\sigma = \alpha\beta + \half \beta d\beta : \pol^2H^2 \to \dga^3$.
Then, as functions $H^2 \otimes \Lambda^2 \kc \to \dga^8$,
\begin{equation}
\label{eq:expand}
\begin{aligned}
& \alpha'(\gamma')^2 - \alpha\gamma^2 -d\left(\beta(\gamma')^2\right) +
d\left(\beta(\half \alpha\beta + \third\beta d\beta) \wedge \gamma\right) \\
= \; &
\alpha\gamma \wedge \sigma + \alpha\sigma^2 + \beta (\alpha')^2 \wedge \gamma'
+ d\left(\beta(\half \alpha\beta + \third\beta d\beta)\right) \wedge \gamma
+ \beta(\half \sigma + \tfrac{1}{12}\beta d\beta) \wedge \alpha^2 \\
= \; & \left(\alpha\sigma +\beta(\alpha')^2 + 
d\left(\beta(\half \alpha\beta + \third\beta d\beta)\right)\right) \wedge \gamma
+ (\half \alpha\sigma + \beta(\alpha')^2 -\half\beta\alpha^2) \wedge \sigma
+ \tfrac{1}{12}\beta d\beta \wedge \alpha^2 .
\end{aligned}
\end{equation}
Expanding a map $H^2 \otimes \kc \to \dga^5$:
\begin{align*}
& \alpha\sigma +\beta(\alpha')^2 + 
d\left(\beta(\half \alpha\beta + \third\beta d\beta)\right) \\
= \; & \alpha\alpha\beta + \half\alpha\beta d\beta
+ \beta(\alpha^2 + \alpha d\beta + (d\beta)^2)
+ d\beta(\half \alpha\beta + \third \beta d\beta)
- \beta\left(\half\alpha d\beta +\tfrac{2}{3}(d\beta)^2\right) \\ 
= \; & \alpha\alpha\beta + \beta\alpha^2 +
\half (\alpha\beta d\beta + \beta \alpha d\beta + d\beta \alpha\beta)
+ \third \left((\beta d\beta)^2 + d\beta \beta d\beta\right)
\end{align*}
The right hand side factors through $\pol^3 H^2$.
Hence the first term on the right hand side of \eqref{eq:expand} factors
through $\rmap$, so that its restriction to $\ndom$ vanishes.

The remaining terms on the right-hand side of \eqref{eq:expand} expand to
\begin{align*}
& \left(\half \alpha\alpha\beta + \quart \alpha \beta d\beta +
\half\beta\alpha^2 + \beta\alpha d\beta + \beta(d\beta)^2\right) \wedge \sigma
+ \tfrac{1}{12}\beta d\beta \wedge \alpha^2 \\
= \; & 
(\half \alpha\alpha\beta + \half\beta\alpha^2) \wedge \sigma +
(\quart \alpha\beta d\beta + \beta\alpha d\beta)\wedge \alpha \beta 
+ \tfrac{1}{12}\beta d\beta \wedge \alpha^2 \\
&+\quart\alpha(\beta d\beta)^2 + \half\beta\alpha d\beta \wedge \beta d\beta
+ \beta(d\beta)^2\wedge \alpha\beta + \half \beta(d\beta)^2 \wedge \beta d\beta
\end{align*}
The first term on the right-hand side factors through $\rmap$ as it is.
The remaining terms can be arranged into 3 groups, each of which factors
through $\rmap$ after addition of an exact term.
\begin{align*}
&(\quart \alpha\beta d\beta + \beta\alpha d\beta)\wedge \alpha \beta 
+ \tfrac{1}{12}\beta d\beta \wedge \alpha^2
+ \tfrac{2}{3}d\left(\beta(\alpha d\beta)\right)^2\\
= &\; - \tfrac{1}{12}(\alpha\alpha\beta+\beta\alpha^2)\wedge \beta d\beta
+ \third(\alpha\beta d\beta + 3\beta\alpha d\beta) \wedge \alpha \beta
+ \tfrac{2}{3}d\beta(\alpha\beta)^2 -
\tfrac{2}{3}\beta\alpha d\beta \wedge \alpha\beta \\
= &\; - \tfrac{1}{12}(\alpha\alpha\beta+\beta\alpha^2)\wedge \beta d\beta
+ \third(\alpha\beta d\beta + \beta\alpha d\beta +d\beta\alpha\beta)
\wedge \alpha \beta
\end{align*}
and
\begin{align*}
&\quart\alpha(\beta d\beta)^2 + \half\beta\alpha d\beta \wedge \beta d\beta
+ \beta(d\beta)^2\wedge \alpha\beta
+ \tfrac{3}{8}d(\beta\alpha\beta\wedge \beta d\beta) \\ 
= \; & 
\quart\alpha(\beta d\beta)^2
+ \tfrac{1}{8}\beta\alpha d\beta \wedge \beta d\beta
+ \quart \beta(d\beta)^2\wedge \alpha\beta
+ \tfrac{3}{8}d\beta\alpha\beta\wedge \beta d\beta \\
= \; & 
\tfrac{1}{8}(\alpha\beta d\beta + \beta \alpha d\beta + d\beta\alpha\beta)
\wedge \beta d \beta
+ \quart\left(d\beta \beta d\beta + \beta(d\beta)^2\right)\wedge \alpha\beta
\end{align*} 
and
\begin{equation*}
 \half \beta(d\beta)^2 \wedge \beta d\beta
+ \tfrac{1}{6} d \left(\beta(\beta d\beta)^2\right)
= \tfrac{2}{3}\left(\beta(d\beta)^2 + d\beta \beta d\beta\right) \wedge \beta d\beta
\end{equation*} 
In summary, $\alpha'(\gamma')^2 - \alpha\gamma^2$ is a sum of exact terms
and terms that factor through $\rmap$, so the restriction to $\ndom$ is exact
as required. This completes the proof of Theorem \ref{thm:special}.
\end{proof}

\subsection{Naturality}
\label{subsec:nat}

Suppose that $f : \dga \rightarrow \dgab$ is a morphism of DGAs, it induces a homomorphism $H^2(f) \otimes \Lambda^2 \pol^2 H^2(f) : H^2(\dga) \otimes \Lambda^2 \pol^2 H^2(\dga) \rightarrow H^2(\dgab) \otimes \Lambda^2 \pol^2 H^2(\dgab)$, which restricts to a map $H^2(\dga) \otimes \Lambda^2 \kc_{\dga} \rightarrow H^2(\dgab) \otimes \Lambda^2 \kc_{\dgab}$. Another restriction of the same map is $\rfunc(H^2(f)) : \rfunc(H^2(\dga)) \rightarrow \rfunc(H^2(\dgab))$. Hence by restricting $H^2(f) \otimes \Lambda^2 \pol^2 H^2(f)$ (or $\rfunc(H^2(f))$) to $\ndom_{\dga}$ we get a homomorphism $\ndom_{\dga} = (H^2(\dga) \otimes \Lambda^2 \kc_{\dga}) \cap \rfunc(H^2(\dga)) \rightarrow \ndom_{\dgab} = (H^2(\dgab) \otimes \Lambda^2 \kc_{\dgab}) \cap \rfunc(H^2(\dgab))$.

\begin{prop} \label{prop:nat}
If $H^2(f) : H^2(\dga) \rightarrow H^2(\dgab)$ is an isomorphism, then for any choice of $\nmt_{\dga}$ there is a choice of $\nmt_{\dgab}$ such that the diagram 
\[
\xymatrix{
\ndom_{\dga} \ar[r]^-{\nmt_{\dga}} \ar[d]_-{\rfunc(H^2(f))} & H^8(\dga) \ar[d]^-{H^8(f)} \\
\ndom_{\dgab} \ar[r]^-{\nmt_{\dgab}} & H^8(\dgab)
}
\]
commutes. 
\end{prop}

\begin{proof}
The map $\nmt_{\dga}$ is determined by the choice of $\alpha_{\dga}$ and $\gamma_{\dga}$ and $\nmt_{\dgab}$ is determined by the choice of $\alpha_{\dgab}$ and $\gamma_{\dgab}$. So it is enough to show that for any choice of $\alpha_{\dga}$ and $\gamma_{\dga}$ we can choose $\alpha_{\dgab}$ and $\gamma_{\dgab}$ such that the diagram
\[
\xymatrix{
H^2(\dga) \otimes \Lambda^2 \kc_{\dga} \ar[rr]^-{\alpha_{\dga} \gamma_{\dga}^2} \ar[d]_-{H^2(f) \otimes \Lambda^2 \pol^2 H^2(f)} & & \dga^8 \ar[d]^-{f} \\
H^2(\dgab) \otimes \Lambda^2 \kc_{\dgab} \ar[rr]^-{\alpha_{\dgab} \gamma_{\dgab}^2} & & \dgab^8
}
\]
commutes, which follows from the commutativity of the diagrams 
\[
\xymatrix{
H^2(\dga) \ar[r]^-{\alpha_{\dga}} \ar[d]_-{H^2(f)} & \clalg^2(\dga) \ar[d]^-{f} & \ar@{}[d]|{\textstyle{\text{and}}} & & \kc_{\dga} \ar[r]^-{\gamma_{\dga}} \ar[d]_-{\pol^2 H^2(f)} & \dga^3 \ar[d]^-{f} \\
H^2(\dgab) \ar[r]^-{\alpha_{\dgab}} & \clalg^2(\dgab) & & & \kc_{\dgab} \ar[r]^-{\gamma_{\dgab}} & \dgab^3
}
\]

Since $H^2(f)$ is an isomorphism, for any section $\alpha_{\dga}$ of the projection $\clalg^2(\dga) \rightarrow H^2(\dga)$ we can choose $\alpha_{\dgab} = f \circ \alpha_{\dga} \circ H^2(f)^{-1}$. Then $\alpha_{\dgab}$ is a section of $\clalg^2(\dgab) \rightarrow H^2(\dgab)$ and the first diagram commutes. 

Since $\pol^2 H^2(f)$ is also an isomorphism, its restriction is an injective map $\pol^2 H^2(f) \big| _{\kc_{\dga}} : \kc_{\dga} \rightarrow \kc_{\dgab}$, and $\kc_{\dgab} = \pol^2 H^2(f)(\kc_{\dga}) \oplus \kc'$ for some subspace $\kc' \leq \kc_{\dgab}$. Given a map $\gamma_{\dga} : \kc_{\dga} \rightarrow \dga^3$ such that $d \gamma_{\dga} = \alpha_{\dga}^2 \big| _{\kc_{\dga}}$ we define $\gamma_{\dgab}$ by $\gamma_{\dgab} \big| _{\pol^2 H^2(f)(\kc_{\dga})} = f \circ \gamma_{\dga} \circ \pol^2 H^2(f)^{-1}$ and setting $\gamma_{\dgab} \big| _{\kc'}$ to be any map such that $d \gamma_{\dgab} \big| _{\kc'} = \alpha_{\dgab}^2 \big| _{\kc'}$. Then $d \gamma_{\dgab} = \alpha_{\dgab}^2 \big| _{\kc_{\dgab}}$ and the second diagram commutes too. 
\end{proof}

\begin{rmk}
\label{rmk:obstruction}
If $H^3(\dga) \cong H^3(\dgab) \cong 0$, then $\nmt_{\dga}$ and $\nmt_{\dgab}$ are independent of choices by Theorem \ref{thm:special}, and any morphism $f : \dga \rightarrow \dgab$ that induces an isomorphism $H^2(f) : H^2(\dga) \rightarrow H^2(\dgab)$ intertwines $\nmt_{\dga}$ and $\nmt_{\dgab}$. It follows that in this case $\nmt$ is invariant under quasi-isomorphisms, in particular it is an obstruction to formality.
\end{rmk}

\section{The pentagonal Massey tensor of closed 8-manifolds}

\subsection{\texorpdfstring{$\nmt$}{P} of spaces}

First we define $\nmt_X$ for a space $X$.

\begin{defn}
Let $X$ be a simply-connected space. Let $\kc_X \leq \pol^2 H^2(X)$ denote the kernel of the cup product $\pol^2 H^2(X) \to H^4(X)$, and let $\ndom_X \leq H^2(X) \otimes \Lambda^2 \kc_X$ be the kernel of the restriction of \eqref{eq:rmap}. We define $\nmt_X : \ndom_X \rightarrow H^8(X)$ to be $\nmt_{\Omega^*_{PL}(X)}$, using the canonical identification between $H^*(X)$ and $H^*(\Omega^*_{PL}(X))$.
\end{defn}

By Theorem \ref{thm:special} $\nmt_X$ is well-defined if $H^3(X) \cong 0$. If $H^3(X) \cong 0$ and $f : \dga \rightarrow \Omega^*_{PL}(X)$ is a quasi-isomorphism for some DGA $\dga$, then by Proposition \ref{prop:nat} the (well-defined) maps $\nmt_A$ and $\nmt_{\Omega^*_{PL}(X)}$ fit into a commutative diagram
\[
\xymatrix{
\ndom_{\dga} \ar[rr]^-{\nmt_{\dga}} \ar[d]_-{\rfunc(H^2(f))} & & H^8(\dga) \ar[d]^-{H^8(f)} \\
\ndom_{\Omega^*_{PL}(X)} \ar[rr]^-{\nmt_{\Omega^*_{PL}(X)}} & & H^8(\Omega^*_{PL}(X))
}
\]
where the vertical maps are isomorphisms. This means that $\nmt_X$ can be computed from any model of $X$.

\begin{prop} \label{prop:nat2}
Let $f : X \rightarrow Y$ be a continuous map between simply-connected spaces such that $H^2(f) : H^2(Y) \rightarrow H^2(X)$ is an isomorphism. Then for any choice of $\nmt_Y$ there is a choice of $\nmt_X$ such that the diagram 
\[
\xymatrix{
\ndom_Y \ar[r]^-{\nmt_Y} \ar[d]_-{\rfunc(H^2(f))} & H^8(Y) \ar[d]^-{H^8(f)} \\
\ndom_X \ar[r]^-{\nmt_X} & H^8(X)
}
\]
commutes. 
\end{prop}

\begin{proof}
Apply Proposition \ref{prop:nat} to the induced map $\Omega^*_{PL}(Y) \rightarrow \Omega^*_{PL}(X)$. 
\end{proof}

\begin{defn} \label{def:nbar}
Let $M$ be a simply-connected closed oriented $8$-manifold with $H^3(M) \cong 0$. We define the canonical element $\nmtb(M) \in \ndom_M^{\vee}$ by $\nmtb(M) = \left< [M], \cdot \right> \circ \nmt_M : \ndom_M \rightarrow \Q$, where $\left< [M], \cdot \right> : H^8(M) \rightarrow \Q$ denotes evaluation on the fundamental class $[M] \in H_8(M)$.
\end{defn}

\subsection{The group \texorpdfstring{$\Theta_8(\bigvee^r S^2)$}{Theta\_8}}

Recall the definition of E-manifolds (Definition \ref{def:eman}). We also recall some definitions and results from \cite{nagy20}, starting with the definition of $\Theta_8(\bigvee^r S^2)$:

\begin{defn} \label{def:theta}
For an integer $r \geq 0$ let 
\[
\textstyle
\Theta_8(\bigvee^r S^2) = \left\{ (M, \varphi) \left| 
\begin{gathered} 
\text{$M$ is a spin $8$-dimensional E-manifold} \\[-4pt] 
\text{$\varphi : \Z^r \rightarrow H_2(M;\Z)$ is an isomorphism} \\[-4pt] 
H_4(M;\Z) \cong 0
\end{gathered} 
\right.
\right\} \Biggm/ \sim
\]
where $(M, \varphi) \sim (M', \varphi')$ if there is an orientation-preserving diffeomorphism $F : M \rightarrow M'$ such that $\varphi' = H_2(F;\Z) \circ \varphi$. The equivalence class of $(M, \varphi)$ will be denoted by $[M, \varphi]$.
\end{defn}

\begin{rmk}
If we identify $\Z^r$ with $H_2(\bigvee^r S^2;\Z)$, then the choice of an isomorphism $\varphi : \Z^r \rightarrow H_2(M;\Z)$ determines a $4$-connected map $\bar{\varphi} : \bigvee^r S^2 \rightarrow M$ (up to homotopy). Conversely, if $M$ is a closed $8$-manifold, then the existence and choice of a $4$-connected map $\bar{\varphi} : \bigvee^r S^2 \rightarrow M$ ensures that $M$ is an E-manifold with $H_4 \cong 0$ and determines an isomorphism $\varphi : \Z^r \rightarrow H_2(M;\Z)$ respectively.
\end{rmk}

$\Theta_8(\bigvee^r S^2)$ is a group under ``connected sum along the $2$-skeleton". If $(M, \varphi)$ represents an element of $\Theta_8(\bigvee^r S^2)$, then $M$ has a unique String structure compatible with its orientation.

\begin{defn}
Let $P_3 = P_3(\bigvee^r S^2)$ denote the third Postnikov-stage of $\bigvee^r S^2$.
\end{defn}

The identification $\Z^r \cong H_2(\bigvee^r S^2;\Z)$ induces identifications $\Z^r \cong H_2(P_3;\Z)$ and $(\Q^r)^{\vee} \cong H^2(P_3)$.

\begin{defn}
Let $\widehat{\Omega}_8^{String}(P_3)$ denote the kernel of the homomorphism $\sigma : \Omega_8^{String}(P_3) \rightarrow \Z$ given by the signature.
\end{defn}

\begin{defn} \label{def:eta}
The map $\eta : \Theta_8(\bigvee^r S^2) \rightarrow \widehat{\Omega}_8^{String}(P_3)$ is given by $\eta([M, \varphi]) = [f]$, where $f : M \rightarrow P_3$ is the composition of the natural map $M \rightarrow P_3(M)$ and the homotopy inverse of $P_3(\bar{\varphi}) : P_3 = P_3(\bigvee^r S^2) \rightarrow P_3(M)$, where $\bar{\varphi} : \bigvee^r S^2 \rightarrow M$ is the map (well-defined up to homotopy) which induces the isomorphism $\varphi : \Z^r = H_2(\bigvee^r S^2;\Z) \rightarrow H_2(M;\Z)$.
\end{defn}

\begin{thm}[{\cite[Proposition 2.2.34]{nagy20}}] \label{thm:eta}
$\eta$ is a well-defined isomorphism.
\qed
\end{thm}

\begin{prop} \label{prop:p3-bord}
$\widehat{\Omega}_8^{String}(P_3) \otimes \Q \cong H_8(P_3)$. 
\end{prop}

\begin{proof}
Let $K_3$ denote the Eilenberg-MacLane space $K(\pi_3(\bigvee^r S^2),3)$. The Postnikov-stage $P_3$ is the total space in a fibration $K_3 \rightarrow P_3 \rightarrow K(\Z^r,2)$, which allows us to compute $\Omega_*^{String}(P_3) \otimes \Q$ and $H_*(P_3)$ using Atiyah--Hirzebruch--Leray--Serre spectral sequences. The natural transformation $\Omega_*^{String}({-}) \otimes \Q \rightarrow H_*$ of homology theories induces a morphism between these spectral sequences. We have $H^*(K_3) \cong \Lambda(H^3(K_3))$, hence $H_i(K_3) \cong (\Lambda(H^3(K_3))_i)^{\vee}$. From \cite[Theorem 3.1.29]{nagy20} we see that the homomorphism $\Omega_i^{String}(K_3) \otimes \Q \rightarrow H_i(K_3)$ is an isomorphism for $i < 8$. Moreover, $\Omega_8^{String}(K_3) \otimes \Q \cong \Omega_8^{String} \otimes \Q \cong \Q$ (detected by the signature) and $H_8(K_3) \cong 0$. Therefore $\Omega_8^{String}(P_3) \otimes \Q \cong H_8(P_3) \oplus \Q$ and $\widehat{\Omega}_8^{String}(P_3) \otimes \Q \cong H_8(P_3)$. 
\end{proof}

By combining Theorem \ref{thm:eta} and Proposition \ref{prop:p3-bord} we get the following: 

\begin{thm} \label{thm:theta-isom}
The map $\Theta_8(\bigvee^r S^2) \otimes \Q \rightarrow H_8(P_3)$, $[M, \varphi] \otimes 1 \mapsto H_8(f)([M])$ is an isomorphism,
where $f : M \rightarrow P_3$ is the map from Definition \ref{def:eta}
and $[M] \in H_8(M)$ is the fundamental class of~$M$.
\qed
\end{thm}

\subsection{Computing \texorpdfstring{$\nmt$}{P} for \texorpdfstring{$\Theta_8(\bigvee^r S^2)$}{Theta\_8}}
\label{subsec:nmt_theta}

In order to understand $\nmt$ for elements of $\Theta_8(\bigvee^r S^2)$, we first
consider it for their third Postnikov stage.

\begin{thm} \label{thm:n-p3}
We have $\ndom_{P_3} = \rfunc(H^2(P_3))$ and the map $\nmt_{P_3} : \ndom_{P_3} = \rfunc(H^2(P_3)) \rightarrow H^8(P_3)$ is an isomorphism.
\end{thm}

\begin{prop}
The Sullivan minimal model of $P_3$ is $(\Lambda V, d^*)$, where $V = V^2 \oplus V^3$, $V^2 \cong H^2(P_3)$, $d^2 = 0$ and $d^3 : V^3 \rightarrow \pol^2 V^2$ is an isomorphism.
\end{prop}

\begin{proof}
The Sullivan minimal model satisfies $V^i \cong \pi_i(P_3)^{\vee} \otimes \Q$, therefore $V^2 \cong \pi_2(\bigvee^r S^2)^{\vee} \otimes \Q$, $V^3 \cong \pi_3(\bigvee^r S^2)^{\vee} \otimes \Q$ and $V^i \cong 0$ for $i \neq 2,3$. Since $P_3$ is simply-connected, $H^2(\Lambda V, d^*) \cong H^2(P_3) \cong \pi_2(\bigvee^r S^2)^{\vee} \otimes \Q$, hence $d^2 = 0$. We have $(\Lambda V)^4 \cong \pol^2 V^2$, so $d^4 = 0$, therefore $H^3(\Lambda V, d^*) \cong \Ker d^3$ and $H^4(\Lambda V, d^*) \cong \Coker d^3$. The Postnikov-stage $P_3$ can be constructed from $\bigvee^r S^2$ by adding cells of dimension $5$ and above to kill homotopy groups in dimensions $4$ and above, so $H^3(P_3) \cong H^4(P_3) \cong 0$. This implies that $d^3$ is an isomorphism. 
\end{proof}

\begin{proof}[Proof of Theorem \ref{thm:n-p3}]
We will use the Sullivan minimal model $(\Lambda V, d^*)$ of $P_3$ to compute $\nmt_{P_3}$. Since $H^4(P_3) \cong 0$, we have $E_{P_3} = \pol^2 H^2(P_3)$. This implies that $\ndom_{P_3} = \rfunc(H^2(P_3))$. We have $\clalg^2(\Lambda V, d^*) = V^2 \cong H^2(\Lambda V, d^*)$, so $\alpha = \Id_{V^2}$ is the only section of the projection $\clalg^2(\Lambda V, d^*) \rightarrow H^2(\Lambda V, d^*)$ (where $V^2$ is identified with $H^2(P_3)$). Since $d^4=0$, we have $\clalg^4(\Lambda V, d^*) = (\Lambda V)^4 = \pol^2 V^2 = \pol^2 H^2(P_3)$, and $\alpha^2 = \Id_{\pol^2 H^2(P_3)}$. We also saw that $d^3$ is an isomorphism, so $\gamma = (d^3)^{-1} : E_{P_3} = \pol^2 H^2(P_3) \rightarrow V^3$ is the only map satisfying $\alpha^2 = d^3\gamma$. 

We have 
\[
\begin{aligned}
(\Lambda V)^7 &= \pol^2 V^2 \otimes V^3 \\
(\Lambda V)^8 &= \pol^4 V^2 \oplus (V^2 \otimes \Lambda^2 V^3) \\
(\Lambda V)^9 &= (\pol^3 V^2 \otimes V^3) \oplus \Lambda^3 V^3 
\end{aligned}
\]
The differential $d^7$ is surjective onto $\pol^4 V^2$ (because the multiplication map $\pol^2 V^2 \otimes \pol^2 V^2 \rightarrow \pol^4 V^2$ is surjective), $d^8 \big| _{\pol^4 V^2} = 0$ and $d^8 \big| _{V^2 \otimes \Lambda^2 V^3} : V^2 \otimes \Lambda^2 V^3 \rightarrow \pol^3 V^2 \otimes V^3$ is the map \eqref{eq:rmap} (when $V^3$ is identified with $\pol^2 V^2$ via $d^3$). This shows that $\clalg^8(\Lambda V, d^*) = \pol^4 V^2 \oplus \Ker d^8 \big| _{V^2 \otimes \Lambda^2 V^3}$ and $H^8(P_3) = \clalg^8(\Lambda V, d^*) / \pol^4 V^2 \cong \Ker d^8 \big| _{V^2 \otimes \Lambda^2 V^3} \cong \rfunc(V^2)$. Moreover, $\alpha\gamma^2 : H^2(P_3) \otimes \Lambda^2\pol^2 H^2(P_3) \rightarrow (\Lambda V)^8$ is the inclusion of the second component (using the identifications $V^2 \cong H^2(P_3)$ and $V^3 \cong \pol^2 H^2(P_3)$). Therefore $\nmt_{P_3} : \rfunc(H^2(P_3)) \rightarrow H^8(P_3)$, the map induced by the restriction of $\alpha\gamma^2$, is an isomorphism. 
\end{proof}

Since $H_8(P_3) \cong H^8(P_3)^{\vee}$, we immediately get the following:

\begin{cor} \label{cor:n-dual}
The map $H_8(P_3) \rightarrow \rfunc(H^2(P_3))^{\vee}$ given by $x \mapsto \left< x, \cdot \right> \circ \nmt_{P_3} : \rfunc(H^2(P_3)) \rightarrow \Q$ is an isomorphism.
\qed
\end{cor}

Now we can prove the main result of this section:

\begin{thm}
The map $\Theta_8(\bigvee^r S^2) \otimes \Q \rightarrow \rfunc((\Q^r)^{\vee})^{\vee}$ given by $[M, \varphi] \otimes 1 \mapsto ((\varphi_{\Q}^{\vee})^{-1})^*(\nmtb(M))$ is an isomorphism, where $\nmtb(M) \in \ndom_M^{\vee} = \rfunc(H^2(M))^{\vee}$ is the canonical element defined in Definition \ref{def:nbar} and $((\varphi_{\Q}^{\vee})^{-1})^* : \rfunc(H^2(M))^{\vee} \rightarrow \rfunc((\Q^r)^{\vee})^{\vee}$ is the isomorphism induced by the inverse of $\varphi_{\Q}^{\vee} : H^2(M) \cong H_2(M)^{\vee} \rightarrow (\Q^r)^{\vee}$, which is the dual of $\varphi_{\Q} = \varphi \otimes \Id_{\Q} : \Q^r \rightarrow H_2(M)$.
\end{thm}

As the canonical element $\nmtb(M)$ is determined by $\nmt_M$ and $((\varphi_{\Q}^{\vee})^{-1})^*$ is just the identification of $\rfunc(H^2(M))^{\vee}$ with the standard space $\rfunc((\Q^r)^{\vee})^{\vee}$ coming from $\varphi$, this means that $\nmt$ is a complete invariant of $\Theta_8(\bigvee^r S^2) \otimes \Q$.

\begin{proof}
We will show that this map is the composition of the isomorphisms in Theorem \ref{thm:theta-isom} and Corollary \ref{cor:n-dual} (if $H^2(P_3)$ is identified with $(\Q^r)^{\vee}$). The composition maps $[M, \varphi] \otimes 1$ to $\left< H_8(f)([M]), \cdot \right> \circ \nmt_{P_3}$. By the naturality of the cap product and of $\nmt$ (see Proposition \ref{prop:nat2}) we have 
\[
\left< H_8(f)([M]), \cdot \right> \circ \nmt_{P_3} = \left< [M], \cdot \right> \circ H^8(f) \circ \nmt_{P_3} = \left< [M], \cdot \right> \circ \nmt_M \circ \rfunc(H^2(f)) = \nmtb(M) \circ \rfunc(H^2(f)) \text{\,.} 
\]
Since $H_2(f;\Z) \circ \varphi$ is the fixed identification $\Z^r \cong H_2(P_3;\Z)$ (see Definition \ref{def:eta}), this element is identified with $\nmtb(M) \circ \rfunc((\varphi_{\Q}^{\vee})^{-1}) = ((\varphi_{\Q}^{\vee})^{-1})^*(\nmtb(M))$ when $H^2(P_3)$ is identified with $(\Q^r)^{\vee}$.
\end{proof}

\subsection{Fourfold Massey products on 8-dimensional E-manifolds}
\label{subsec:fourfold}

We now consider the relation of $\nmt$ to fourfold Massey products.
Recall that more generally, for $x_1, x_2, x_3, x_4 \in H^*(X)$
the fourfold Massey product is defined if $x_1x_2 = x_2x_3 = x_3x_4 = 0$
and the triple products $\gen{x_1, x_2, x_3}$ and $\gen{x_2, x_3, x_4}$ vanish
simultaneously, in the sense that it  
is possible to choose representatives $\alpha_i \in \calz^*$ and prederivatives
$\gamma_{12}$, $\gamma_{23}$, $\gamma_{34} \in \dga^*$ for $\alpha_1\alpha_2$,
$\alpha_2\alpha_3$ and $\alpha_3\alpha_4$ such that
$\alpha_1\gamma_{23} + \gamma_{12}\alpha_3$ and
$\alpha_2\gamma_{34} + \gamma_{23}\alpha_4$ are both exact. 
Then the fourfold product $\gen{x_1, x_2, x_3, x_4}$ is represented by
\begin{equation}
[\alpha_1\sigma_2 + \gamma_{12}\gamma_{34} + \sigma_1 \alpha_4] \in H^*(X)
\end{equation}
whenever
\[ d\sigma_1 = \alpha_1\gamma_{23} + \gamma_{12}\alpha_3, \qquad
d\sigma_2 = \alpha_2\gamma_{34} + \gamma_{23}\alpha_4 . \]
If $x_5 \in H^*$ such that also $x_4x_5 = x_5x_1 = 0$, then $\gen{x_1, x_2, x_3, x_4}x_5$ is closely related to the evaluation of $\nmt$ on certain elements of $\ndom$. Recall from Definition \ref{def:penta} that the multilinear map $\cmap : H^2(X)^{\times 5} \rightarrow \rfunc(H^2(X)) \leq H^2(X) \otimes \Lambda^2 \pol^2 H^2(X)$ is given by the formula
\begin{equation}
\label{eq:ordinary}
\cmap(x_1, x_2, x_3, x_4, x_5) = \sum_{\text{cyc}} x_1 \otimes (x_2x_3 \wedge x_4x_5) \text{\,.}
\end{equation}

\begin{defn}
\label{def:ordinary}
We will call an element of $\ndom_X \leq \rfunc(H^2(X))$ \emph{ordinary} if it can be obtained as $\cmap(x_1, x_2, x_3, x_4, x_5)$ for some $x_1, x_2, x_3, x_4, x_5 \in H^2(X)$ with $x_1x_2 = x_2x_3 = x_3x_4 = x_4x_5 = x_5x_1 = 0 \in H^4(X)$.
\end{defn}

In general, the relation between the fourfold Massey products and the
evaluation of $\nmt$ on ordinary elements of $\ndom$ is complicated by the
dependence of both objects, and in particular the fourfold Massey product, on
choices.
We are therefore content to restrict attention to the context of 8-dimensional
E-manifolds $M$. Then for $x_1, x_2, x_3, x_4 \in H^2(M)$, the only choices
in the definition of $\gen{x_1, x_2, x_3, x_4} \in H^6(M)$ that affect the
result are those of $\sigma_1$ and $\sigma_2 \in \dga^4$.
Adding a closed form to either of
those has the effect of changing $\gen{x_1, x_2, x_3, x_4}$ by
an element of $x_1H^4(M) + x_4H^4(M)$. In particular, if $x_5 \in H^2(M)$
has $x_5x_1 = x_4x_5 = 0$, then
\[ \gen{x_1, x_2, x_3, x_4}x_5 \in H^8(M) \cong \Q \]
is independent of all choices.
Moreover, by Poincar\'e duality the set of possible values of $\gen{x_1, x_2, x_3, x_4}$ is completely
determined by knowing the values of $\gen{x_1, x_2, x_3, x_4}x_5 \in H^8(M)$
for all such $x_5$. Those are determined by evaluating $\nmt$ on the
ordinary elements: 

\begin{prop}
For any space $X$ with $H^3 = H^5 = 0$ and any $x_1, x_2, x_3, x_4, x_5 \in H^2(X)$ such that $x_ix_{i+1} = 0 \in H^4(X)$,
\[ \gen{x_1,x_2,x_3,x_4}x_5 = \nmt\left(\cmap(x_1, x_2, x_3, x_4, x_5)\right)
\in H^8(X) . \]
\end{prop}

\begin{proof}
Since the LHS is independent of the choices, it suffices to prove the claim
if we pick  $\alpha : H^2 \to \calz^2$
and $\gamma : \kc \to \dga^3$ as in the definition of $\nmt$ and
take $\alpha_i = \alpha(x_i)$, $\gamma_{i,i+1} = \gamma(x_i x_{i+1})$
in the definition of $\gen{x_1, x_2, x_3, x_4}$ (while choosing $\sigma_1, \sigma_2$ arbitrarily).
Now
\begin{align*}
[\alpha_1\sigma_2 + \gamma_{12}\gamma_{34} + \sigma_1 \alpha_4]x_5
&=
[(d\gamma_{51})\sigma_2 + \gamma_{12}\gamma_{34}\alpha_5 + \sigma_1d\gamma_{45}]
\\ &=
[-\gamma_{51}d\sigma_2 + \gamma_{12}\gamma_{34}\alpha_5 + (d\sigma_1)\gamma_{45}]
\\ &=
[(\alpha_2\gamma_{34} + \gamma_{23}\alpha_4)\gamma_{51}
+ \gamma_{12}\gamma_{34}\alpha_5 + 
(\alpha_1\gamma_{23} + \gamma_{12}\alpha_3)\gamma_{45}]
\\ &= \left[ \sum_{\text{cyc}}\alpha_1\gamma_{23}\gamma_{45} \right]
\end{align*}
which equals $\nmt$ evaluated on \eqref{eq:ordinary}.
\end{proof}

In summary, for any 8-dimensional E-manifold the fourfold Massey products
are completely determined by $\nmt$. The reverse holds too if $\ndom$ is
generated by ordinary elements, which is in fact the case for
elements of $\Theta_8(\bigvee^r S^2)$. More generally, if the cup product
$H^2 \times H^2 \to H^4$ is trivial, then $\ndom$ is simply $\rfunc(H^2)$ and
every element in the image of $\cmap$ is ordinary in the sense of
Definition \ref{def:ordinary}, so the next lemma implies that $\nmt$ is
determined by the fourfold Massey products.

\begin{lem}
\label{lem:ord}
For any finite-dimensional vector space $V$, $\rfunc(V)$
is generated by the image of $\cmap : V^{\times 5} \to \rfunc(V)$.
\end{lem}

\begin{proof}
This amounts to proving surjectivity of the linear map
$V^{\otimes 5} \to \rfunc(V)$ induced by $\cmap$.
Since that is a non-zero $GL(V)$-equivariant map and $\rfunc(V)$ is
irreducible by Lemma \ref{lem:ndom}, it is surjective by Schur's lemma.
\end{proof}

\section{General definition of \texorpdfstring{$\nmt$}{P}}
\label{sec:general}

The definition and proof of well-definedness of $\nmt$ in general works
much the same as in the 8-dimensional case, once we have set up the appropriate
notation.
However, in the discussion of the 8-dimensional case we used the
hypothesis $H^3 = 0$ to completely eliminate dependence of $\nmt$
on choices of cochains. Without such a simplifying hypothesis we will need
to describe a transformation rule for how $\nmt$ depends on the choices.
The transformation rule is described in terms of uniform triple Massey
products, and we will therefore also need to describe those.

\subsection{Products of algebra-valued maps}

For a graded algebra $\dga$, graded vector spaces $V$ and~$W$, and linear maps
$\alpha : V \to \dga$ and $\beta : W \to \dga$, we can define an obvious
product $\alpha \beta : V \otimes W \to \dga$.
But we also wish to introduce some notions of products of graded symmetric
and antisymmetric maps.

For a graded vector space $V$, let $\gsym^p V$ and $\agsym^p V$ denote the
quotients of $V^{\otimes p}$ by relations of graded symmetry and graded
antisymmetry respectively.
Thus a linear map $\alpha : \gsym^p V \to \dga$ is essentially the same
as a multilinear map $\alpha : V^p \to \dga$ such that
\[ \alpha(x_{\sigma(1)}, \ldots, x_{\sigma(p)})
= (-1)^{\sum d_i d_j}
\alpha(x_1, \ldots, x_p) \]
for any permutation $\sigma$ and any $x_1, \ldots, x_p \in V$ of degrees
$d_1, \ldots, d_p$,
where the sum is taken over all $i < j$ such that $\sigma(i) > \sigma(j)$.
Meanwhile a linear map $\alpha : \agsym^p V \to \dga$ is equivalent to an
$\alpha : V^p \to \dga$ such that
\[ \alpha(x_{\sigma(1)}, \ldots, x_{\sigma(p)})
= (\sgn \sigma)(-1)^{\sum d_i d_j} \alpha(x_1, \ldots, x_p)
= (-1)^{\sum (1+d_i d_j)} \alpha(x_1, \ldots, x_p) . \]
Given $\alpha : \gsym^p V \to \dga$ and $\beta : \gsym^q V \to \dga$,
of degree $r$ and $s$ respectively (\ie $\alpha$ maps the degree $d$ part of
the graded vector space $\gsym^p V$ to $\dga^{d+r}$), we define
a product $\alpha\beta : \gsym^{p+q}V \to \dga$ (of degree $r+s$) as follows.
\[ (\alpha\beta)(x_1, \ldots, x_{p+q}) =
\frac{1}{p!q!}\sum_{\sigma \in S_{p+q}}(-1)^{s(d_1 + \cdots d_p) + \sum d_id_j}
\alpha(x_{\sigma(1)}, \ldots, x_{\sigma(p)})
\beta(x_{\sigma(p+1)}, \ldots, x_{\sigma(p+q)})
\]
\begin{rmk*}
The notation is a bit ambiguous in that $\alpha\beta$
could be interpreted as being defined on $\gsym^pV \otimes \gsym^qV$
or $\gsym^{p+q}V$, or also on $\gsym^2 \gsym^p V$ if $p = q$.
Rather than introducing more cumbersome notation for the products themselves,
we aim to disambiguate by specifying the domains clearly.
\end{rmk*}

The symmetric product satisfies convenient versions of associativity, graded
commutativity and (if $\dga$ is a DGA) the Leibniz rule:
\begin{align*}
\alpha(\beta\gamma) &= (\alpha\beta)\gamma : \gsym^{p+q+r}V \to \dga \\
\alpha\beta &= (-1)^{rs}\beta\alpha : \gsym^{p+q}V \to \dga\\
d(\alpha\beta) &= (d\alpha)\beta + (-1)^r \alpha(d\beta) : \gsym^{p+q}V \to \dga
\end{align*}
For $\alpha : \agsym^p V \to \dga$ and $\beta : \agsym^q V \to \dga$,
we define $\alpha \wedge \beta : \agsym^{p+q}V \to \dga$ analogously, \ie
\begin{align*}
&(\alpha \wedge \beta)(x_1, \ldots, x_{p+q}) \\
=&
\frac{1}{p!q!}\sum_{\sigma \in S_{p+q}}(\sgn \sigma)(-1)^{s(d_1 + \cdots d_p) + \sum d_id_j}
\alpha(x_{\sigma(1)}, \ldots, x_{\sigma(p)})
\beta(x_{\sigma(p+1)}, \ldots, x_{\sigma(p+q)}) .
\end{align*}
This product enjoys the same associativity property and
Leibniz rule as the graded symmetric one, and it is bigraded commutative in the
sense that
\[ \alpha\wedge\beta = (-1)^{pq+rs}\beta\wedge\alpha : \gsym^{p+q}V \to \dga .\]
Finally, given $\alpha : V \to \dga$ of degree $r$, the map
$(x_1, \ldots, x_p) \mapsto \alpha(x_1)\cdots\alpha(x_p)$
is itself graded symmetric or antisymmetric, depending on whether $r$
is even or odd.
We denote the resulting map $\gsym^pV \to \dga$ or $\agsym^pV \to \dga$
by $\alpha^p$. Note that with this convention \eg $6\alpha^3$ is
$\alpha\alpha\alpha$ or $\alpha \wedge \alpha \wedge \alpha$; 
this convention is the right one if one wants the option to apply the set-up
to free abelian groups rather than vector spaces.

\subsection{Uniform triple products and the \bmp}
\label{subsec:triples}

The dependence of $\nmt$ on choices will be expressed in terms of the
uniform triple products in the sense of \cite[\S2.3]{bmp}, which we now recall.

Let $\dga$ be a DGA with homology $H^*$. As before, $\clalg^* \subseteq \dga$ denotes the subalgebra of closed elements. Let $E^* \subseteq \gsym^2 H^*$ denote the kernel of the product map $\gsym^2 H^* \rightarrow H^*$.

\begin{defn}
\label{def:choice}
A \emph{cochain choice} is a pair
$c = (\alpha,\gamma)$ where $\alpha : H^* \to \clalg^*$ is a right inverse for
the projection to cohomology, and $\gamma : E^* \to \dga^{*-1}$ satisfies
$d\gamma = \alpha^2_{|E^*}$.
\end{defn}

Given two choices $c$ and $c'$, pick $\beta : H^* \to \dga^{*-1}$
such that $d\beta = \alpha' - \alpha$.
Then $\gamma' - \gamma - \beta(\alpha+\half d\beta)$
maps $\kc^* \to \clalg^{*-1}$, so induces a map
$\delta(c',c) : \kc^* \to H^{*-1}$.

This $\delta$ can in turn depend on the choice of $\beta$.
If we set $\beta' = \beta + \eta$ for some $\eta : H^* \to \clalg^{*-1}$,
then $\beta'(\alpha+\half d\beta') - \beta(\alpha+\half d\beta)
= \half\eta(\alpha+\alpha')$, so induces the map
$[\eta]\idf : \gsym^2H^* \to H^{*-1}$.
But in any case, if we let
\begin{itemize}
\item $L_1 = \hom(H^*, H^{*-1})$, the space of degree $-1$ maps $H^* \to H^{*-1}$
\item $L_2 = \hom(\kc^*, H^{*-1})$
\end{itemize}
and $\idf L_1 \subseteq L_2$ the subspace of maps $\kc^* \to H^{*-1}$ that are
restrictions of the symmetric product
$\idf \, n : {\gsym^2 H^* \rightarrow H^{*-1}}$ of $\idf : H^* \to H^*$ and some
$n : H^* \to H^{*-1}$,
then $\delta$ is well-defined in $L_2/\idf L_1$.

\begin{defn}
\label{def:choices}
Say that two choices $c$ and $c'$ are \emph{equivalent} if
$\delta(c',c) = 0 \in L_2/\idf L_1$. The set of equivalence classes $\calc$ then
forms an affine vector space modelled on $L_2/\idf L_1$.
\end{defn}

In particular, if $c$ and $c'$ are two cochain choices, and their equivalence classes in $\calc$ are also denoted by $c$ and $c'$, then the element $c'-c \in L_2/\idf L_1$ is represented by $\delta(c',c) \in L_2$ (with any choice of $\beta$ used in the definition of $\delta(c',c)$).

Now let $\tripsp$ denote the kernel of the full graded symmetrisation map
\begin{equation}
\label{eq:fullsym}
s : H^* \otimes \kc^* \to \gsym^3H^* .
\end{equation}

\begin{defn}
Given a choice $c$, we define the \emph{uniform triple product}
to be the linear map
\[ \trip_c : \tripsp^* \to H^{*-1} \]
induced by the restriction of
\[ \alpha\gamma : H^* \otimes \gsym^2 H^* \to \dga^{*-1} . \]
\end{defn}

The effect of changing the choice can be expressed as the restriction of the
product of the identity map
$\idf : H^* \to H^*$ with $\delta : \kc^* \to H^{*-1}$
(see \cite[discussion following (9)]{bmp}):
\begin{equation}
\label{eq:uni_rule}
\trip_{c'} - \trip_c = \idf\, \delta(c',c) : \tripsp^* \to H^{*-1} .
\end{equation}
In particular, choices that are equivalent in the sense of
Definition \ref{def:choices} give the same triple product~$\trip_c$,
and $\trip$ is an affine linear map $\calc \to \hom(\tripsp^*, H^{*-1})$.

\smallskip
The uniform triple product is closely related to the \bmp. 
Let $\bmsp^*(\dga) = \kersym{\gsym^2 \kc^*}$, the kernel of the
full graded symmetrisation map $\gsym^2 \kc^* \to \gsym^4 H^*(\dga)$. 

\begin{defn}
The \bmp{} is the degree $-1$ linear map
$\princ : \bmsp^{*+1}(\dga) \to H^*(\dga)$
induced by $\alpha^2 \cdot \gamma : \bmsp^{*+1} \to \dga^*$ for any
cochain choice $(\alpha, \gamma)$.
\end{defn}

Although the definition of the Bianchi-Massey tensor involves the cochain
choice, it is in fact completely independent of that choice
\cite[Lemma 2.5]{bmp}.
For Poincar\'e DGAs it is equivalent to the uniform triple product.

\begin{lem}[{\cite[Lemma 2.8]{bmp}}]
\label{lem:trip_bmp}
For an $m$-dimensional Poincar\'e DGA $\dga$, the top degree component
$\princ : \bmsp^{m+1}(\dga) \to H^m(\dga)$ of the \bmp{} is equivalent to
the uniform triple product. In particular, there is a cochain choice $c$ such
that $\trip_c = 0$ if and only if $\princ = 0$.
\end{lem}

\subsection{Defining \texorpdfstring{$\nmt$}{P}} \label{ss:gen-def}

We now present the definition of $\nmt$ for a general DGA $\dga$.
Let us first describe its domain.
For any graded vector space $H^*$, let
\begin{equation}
\label{eq:grmap}
\begin{aligned}
\rmap: H^* \otimes \agsym^2\gsym^2 H^* &\to \gsym^3 H^* \otimes \gsym^2 H^* \\
q \otimes (x_1x_2 \wedge x_3x_4) &\mapsto
(qx_1x_2)(x_3x_4) - (-1)^{(d_1+d_2)(d_3+d_4)}(qx_3x_4)(x_1x_2)
\end{aligned}
\end{equation}
for $x_i$ of degree $d_i$.

\begin{defn}
\label{def:ndom_gen}
For a graded algebra $H^*$, if $\kc^* \subset \gsym^2 H^*$ denotes the kernel of
the product map $\gsym^2 H^* \to H^*$, then we define $\ndom^*(H^*)$ to be the kernel
of the restriction of $\rmap$ to $H^* \otimes \agsym^2 \kc^*$.
\end{defn}

\begin{lem}
Let $\dga$ be a DGA with cohomology $H^* = H^*(\dga)$. Given a choice $c = (\alpha, \gamma)$,
the restriction of
\[ \alpha\gamma^2 : H^* \otimes \agsym^2 \gsym^2 H^* \to \dga^{*-2} \]
to $\ndom = \ndom(H^*)$ takes closed values.
\end{lem}

\begin{proof}
Like in \S\ref{subsec:special_proof} we find that
\[ d(\alpha\gamma^2) = \alpha (\alpha^2 \wedge \gamma) =
\third \alpha^3\gamma \circ \rmap : H^* \otimes \agsym^2 \gsym^2 H^* \to \dga^{*-1} \]
and thus vanishes on $\ndom$.
\end{proof}

\begin{defn}
Given a choice $c$, the pentagonal Massey tensor 
\[ \nmt_c : \ndom^* \to H^{*-2} \]
is the degree $-2$ linear map induced by $\alpha\gamma^2$.
\end{defn}

We now wish to understand the dependence of $\nmt_c$ on $c$.
Consider the natural inclusion
\begin{equation}
\label{eq:jinc}
j : \agsym^2\gsym^2 H^* \to \gsym^2 H^* \otimes \gsym^2 H^*, \; 
x\wedge y \mapsto xy - (-1)^{d_1d_2}yx .
\end{equation}

\begin{lem}
\label{lem:rmap_trip}
The image of $\ndom$ under $\idf \, j :
H^* \otimes \agsym^2\gsym^2 H^* \to H^* \otimes \gsym^2 H^* \otimes \gsym^2 H^*$
is contained in $\tripsp \otimes E^*$ (where $\tripsp$ denotes the kernel
of the full symmetrisation map $s$ from \eqref{eq:fullsym} like before).
\end{lem}

\begin{proof}
$\rmap$ is the composition of $\idf \, j$
and $s\, \idf :
H^* \otimes \gsym^2 H^* \otimes \gsym^2 H^* \to \gsym^3 H^* \otimes \gsym^2 H^*$,
so $\ker \rmap$ is mapped to $(\ker s) \otimes \gsym^2 H^*$.
\end{proof}

For any linear map $\delta : \kc^* \to H^{*-1}$ Lemma \ref{lem:rmap_trip} means that the composition 
\[ 
(\trip_c \delta) \circ (\idf \, j) : \ndom^* \to H^{*-2}
\]
of $\idf j : \ndom^* \rightarrow \tripsp \otimes E^*$ and $\trip_c \delta : \tripsp \otimes E^* \rightarrow H^{*-2}$
is well-defined.

\begin{thm}
\label{thm:choices}
Let $c = (\alpha, \gamma)$ and $c' = (\alpha', \gamma')$
be two different choices, and let $\delta : \kc^* \to H^{*-1}$ be
any representative of $c' - c \in L_2/\idf L_1$. Then
\begin{equation}
\label{eq:choice_dep}
\nmt_{c'} - \nmt_c = (\trip_c\delta)\circ (\idf \, j) + \idf \, \delta^2 . 
\end{equation}
(In particular the right hand side is independent of the choice of
representative $\delta$ of $c'-c$.)
\end{thm}

As $\nmt_{c+\delta+\epsilon} - \nmt_{c+\delta} - \nmt_{c+\epsilon} + \nmt_c
= \idf \, \delta {\wedge} \epsilon$ is bilinear in $\delta$ and $\epsilon$ and
independent of $c$, we can interpret the relation \eqref{eq:choice_dep}
to mean that $c \mapsto \nmt_c$ is
an ``affine quadratic function'' $\calc \to \hom(\ndom^*, H^{*-2})$.

If $c$, $c'$ and $c''$ are three different choices, it is not immediately
apparent that the expressions for $\nmt_{c''}-\nmt_{c'}$ and $\nmt_{c'}-\nmt_c$
given by \eqref{eq:choice_dep} add up to the expression for $\nmt_{c''}-\nmt_c$.
As a sanity check, let us verify this directly.
Suppose $c''-c'$ is represented by $\epsilon : \kc^* \to H^{*-1}$.
Then we want vanishing of
\begin{align*}
&\trip_c(\delta+\epsilon)\circ (\idf \, j)  + \idf \, (\delta+\epsilon)^2 -
\trip_{c'}\epsilon \circ (\idf \, j)  - \idf\,\epsilon^2
- \trip_c \delta \circ (\idf \, j)  - \idf\, \delta^2 \\
= \; &
(\trip_c - \trip_{c'})\epsilon \circ (\idf \, j) + \idf \, \delta {\wedge} \epsilon \\
= \; & (\trip_c - \trip_{c'} + \idf\,\delta)\epsilon \circ (\idf \, j)
\end{align*}
which does indeed vanish by \eqref{eq:uni_rule}.

\begin{proof}[Proof of Theorem \ref{thm:choices}]
We prove the claim in three steps.
\begin{itemize}
\item If $\alpha' = \alpha$, then \eqref{eq:choice_dep} holds for some
representative $\delta : \kc^* \to H^{*-1}$ of $c' - c$.
\item The right-hand side of \eqref{eq:choice_dep} 
is independent of the choice of representative of $c'- c$.
\item Given $\alpha, \gamma$ and $\alpha'$, there is some choice of $\gamma'$
such that $c'-c$ is represented by $\delta = 0$ and $\nmt_{c'} = \nmt_c$.
\end{itemize}
So let us first consider the case when $\alpha' = \alpha$.
Then $\gamma' = \gamma + \eta$ for some $\eta : \kc^* \to \clalg^{*-1}$,
and
\[ \alpha(\gamma')^2 - \alpha\gamma^2
= \alpha(\gamma\wedge \eta + \eta^2)
= (\alpha\gamma\eta) \circ (\idf j) + \alpha\eta^2 . \]
The restriction of this to $\ndom$ induces precisely the map in
\eqref{eq:choice_dep}, for $\delta : \kc^* \to H^{*-1}$ the map induced by
$\eta$.

Next we show that both terms in \eqref{eq:choice_dep} vanish when
$\delta$ is the restriction of $n\idf : \gsym^2 H^* \to H^{*-1}$ for some
$n : H^* \to H^{*-1}$.
Writing $\idf(n\idf)^2 : H^* \otimes \agsym^2 \gsym^2 H^* \to H^{*-2}$ as
$\half(\idf(n\idf) + \half n\idf^2) \wedge n\idf - \frac{1}{4} n\idf^2 \wedge n\idf$,
the restriction of the first term to $\ndom = \ker \rmap$ vanishes because
$\idf(n\idf) + \half n\idf^2$ is fully symmetric, while the restriction to
$\ndom \subseteq H^* \otimes \agsym^2 \kc^*$ of the
second term vanishes because ${\idf^2}_{|\kc^*} = 0$ by the definition of $\kc^*$.

Given a choice $(\alpha, \gamma)$, and some $\eta : H^* \to \clalg^{*-1}$
representing $n$, $(\trip_c (n\idf)) \circ (\idf j) : \ndom^* \to H^{*-2}$ is induced
by the restriction of
\[ \alpha\gamma \wedge \eta\alpha : H^* \otimes \agsym^2 \kc^* \to \dga^{*-2} . \]
Now
\[ \alpha \gamma \wedge \eta \alpha + d(\eta\gamma^2)
= \alpha\eta\alpha\wedge \gamma + \eta\alpha^2\wedge\gamma, \]
whose restriction to $\ndom$ vanishes. Hence 
$(\trip_c (n\idf)) \circ (\idf j)$ vanishes too. 

Finally, given $\alpha$, $\gamma$ and $\alpha'$, 
write $\alpha' = \alpha + d\beta$ for some $\beta : H^2 \to \dga^1$,
and take $\gamma' = \gamma + \alpha\beta + \half\beta d \beta$.
Then certainly $c' - c$ is represented by $\gamma = 0$.
Meanwhile, with the notation we have set up we can show that
the restriction of $\alpha'(\gamma')^2 - \alpha\gamma^2$ to $\ndom$ is exact
using exactly the same calculations as in \eqref{eq:expand} and the following
equations. 
\end{proof}

\section{The role of \texorpdfstring{$\nmt$}{P} in rational homotopy classification}

We expect that for sufficiently highly connected closed manifolds,
$\nmt$ should play a role in the rational homotopy classification alongside
the \bmp, but even saying what it means for the pentagonal Massey tensors of
two spaces to ``be the same'' is somewhat subtle because of the way that $\nmt$
depends on cochain choices.

Any rational homotopy classification result will encompass as a special case
a criterion for formality. In the present paper we are content to just prove
such a formality criterion, in the context of closed \tkc{} manifolds
of dimension $5n-2$. But even to state that we first need to at least clarify
what it means for $\nmt$ to vanish.

\subsection{Dependence of \texorpdfstring{$\nmt$}{P} on choices}

While the transformation rule for how $\nmt$ depends on choices is given in
Theorem \ref{thm:choices}, we would now like to clarify how the terms in the
formula behave, and in particular in what sense $\nmt$ is independent
of choices if the \bmp{} is trivial.

\begin{lem}
\label{lem:same_trip}
Suppose that $\delta : \kc^* \to H^{*-1}$ is such that the restriction
of $\idf\delta : H^* \otimes E^* \to H^{*-1}$ to $\tripsp$ vanishes.
Then the restriction of $\idf\delta^2 : H^* \otimes \agsym^2 \kc^* \to H^{*-2}$
to $\ndom$ vanishes, while
$\trip_c\delta \circ (\idf j) : \ndom^* \to H^{*-2}$ is independent of the choice $c$.
\end{lem}

\begin{proof}
$\idf \delta^2$ can be factorised as $((\idf \delta)\delta) \circ \idf j$
for $j$ as in \eqref{eq:jinc}, so vanishes on $\ndom$ by
Lemma \ref{lem:rmap_trip}.

Now suppose $c$ and $c'$ are two different choices. Then
according to \eqref{eq:uni_rule} $\trip_{c'}- \trip_c$
is the restriction to $\tripsp$ of $\idf \delta' : H^* \otimes \kc^* \to H^{*-1}$
for some $\delta' : \kc^* \to H^{*-1}$.
We thus need to prove that $((\idf \delta')\delta) \circ (\idf j) = 0$.
But by the definition of $j$,
$((\idf \delta')\delta) \circ (\idf j) 
= ((\idf \delta)\delta') \circ (\idf j)$, whose restriction to $\ndom$ again
vanishes by Lemma \ref{lem:rmap_trip}.
\end{proof}

\begin{cor}
\label{cor:zero_trip}
If $\trip_c = 0$ for some choice $c$, then $\nmt_c$ takes the same value for
all such choices.
\end{cor}

\begin{proof}
If $c$ and $c'$ are two different choices such that $\trip_c = \trip_{c'} = 0$
then the restriction of $\idf \delta(c',c) : H^* \otimes E^* \to H^{*-1}$ to $\tripsp$
vanishes by \eqref{eq:uni_rule}.
Thus Lemma \ref{lem:same_trip} implies that the second 
term in the transformation rule \eqref{eq:choice_dep} for $\nmt$ vanishes,
while the first term vanishes because $\trip_c = 0$.
\end{proof}

\begin{rmk}
\label{rmk:bmt_zero}
If the \bmp{} of a closed manifold vanishes, then there exists some choice
$c$ such that $\trip_c$ vanishes according to Lemma \ref{lem:trip_bmp}.
Thus we can think of $\nmt$ as a single well-defined map $\ndom^* \to H^{*-2}$
in this situation.

This is the case in particular for 8-dimensional E-manifolds.
Indeed, in this case the \bmp{} is vacuous, there is a single equivalence class
$c$ of choice data, the degree 8 component of $\nmt_c$ reduces to the
definition in \S\ref{subsec:special_def}, and Theorem \ref{thm:special}
can be recovered as a special case of Theorem \ref{thm:choices}.
\end{rmk}

\subsection{Formality criterion}
\label{subsec:formality}

For a closed manifold, asking that the Bianchi-Massey and pentagonal Massey
tensors both vanish is a well-defined condition by Remark \ref{rmk:bmt_zero},
and this is clearly a necessary condition for formality
(Remark \ref{rmk:obstruction}). We now prove
Theorem \ref{thm:main1}, \ie that this condition is in fact equivalent to
formality for closed \tkc{} manifolds of dimension up to $5n-2$.

\begin{thm}
\label{thm:formality}
Let $n \geq 2$ and let $M$ be a closed \tkc{} manifold of dimension
$m \leq 5n-2$. Suppose that the \bmp{} $\princ$ of $M$ vanishes, and
that $\nmt_c = 0$ for the cochain choices~$c$ such that $\trip_c = 0$.
Then $M$ is formal.
\end{thm}

For $m \leq 5n-3$, it was already proved in \cite{bmp} that vanishing of
$\princ$ on its own is equivalent to formality (and $\nmt$ vanishes for degree reasons), so really we are only
interested in the case $m = 5n-2$.
To prove formality it is convenient to employ the concept
of \emph{$s$-formality} of
Fern\'andez and Muñoz~\cite{fernandez05}.

\begin{defn}
Let $\Lambda V$ be a minimal Sullivan-algebra, where $V = \bigoplus_i V^i$, $V^i$ denotes the degree~$i$ component, and let $C^i \leq V^i$ be the subspace of closed elements. We say that $\Lambda V$ is $s$-formal if for each $i \leq s$ one can choose a direct complement $N^i \leq V^i$ to $C^i$ such that any closed element in $N^{\leq s} \Lambda V^{\leq s}$ (the ideal in $\Lambda V^{\leq s}$ generated by $N^{\leq s}$) is exact in $\Lambda V$.

A DGA $\dga$ is $s$-formal if its minimal model is isomorphic to an $s$-formal minimal Sullivan-algebra, equivalently, if there is an $s$-formal minimal Sullivan-algebra $\Lambda V$ and a quasi-isomorphism $\psi : \Lambda V \rightarrow \dga$.
\end{defn}

\begin{thm}[{\cite[Theorem 3.1]{fernandez05}}]
Formality is equivalent to $s$-formality
for closed manifolds of dimension $m \leq 2s+1$.
\end{thm}

For $n \geq 2$, it will thus suffice to show that a closed \tkc{} manifold of
dimension up to
$5n-2$ with vanishing \bmp{} and $\nmt$ is $(3n{-}2)$-formal.

\begin{rmk}
\label{rmk:reduce}
To verify $s$-formality of a closed manifold of dimension $m \leq 2s+1$,
it suffices to show that for some minimal model $(\Lambda V, \psi)$ there is a choice of complements $N^i \leq V^i$ such that the closed elements of degree $\leq s$ or equal to
$m$ in $N^{\leq s} \Lambda V^{\leq s}$ are exact, by the following standard
argument.

Let $\wh \clalg^* \subseteq N^{\leq s} \Lambda V^{\leq s}$ be the subalgebra of
closed elements. If $w \in \wh\clalg^s$ with degree $s < i < m$, and the image
of $w$ in $H^i(\dga)$
is non-zero then by Poincar\'e duality there is a class $x \in H^{m-i}(\dga)$ such that $x[\psi(w)]$ is non-zero. This $x$ must be the image of some
$u \in \Lambda V^{\leq s}$. Then $uw \in \wh\clalg^m$ and its image in $H^m$
is non-zero. 
\end{rmk}

\begin{proof}[Proof of Theorem \ref{thm:formality}]
By the assumptions (and Lemma \ref{lem:trip_bmp}) there is a cochain choice $c = (\alpha,\gamma)$ such that $\trip_c = 0$ (\ie $\alpha\gamma$ takes exact values on $\tripsp$) and $\nmt_c = 0$. We will fix such a $c$.

We build a minimal model $(\Lambda V, \psi)$ recursively as follows. We define $C^{i+1}$ and $N^i$ to be the cokernel and kernel of the map $\clalg^{i+1}(\Lambda V^{<i}) \rightarrow H^{i+1}(\dga)$ induced by $\psi$, and we set $V^i = C^i \oplus N^i$. The differential is given by $d_{|C^{i+1}} = 0$ and $d_{|N^i} = \Id_{N^i} : N^i \rightarrow \clalg^{i+1}(\Lambda V^{<i}) \leq \Lambda V$. We choose $\psi_{|C^{i+1}} : C^{i+1} \rightarrow \clalg^{i+1}(\dga)$ such that the image of the composition $C^{i+1} \rightarrow \clalg^{i+1}(\dga) \rightarrow H^{i+1}(\dga)$ is a complement to the image of $\clalg^{i+1}(\Lambda V^{<i}) \rightarrow H^{i+1}(\dga)$ and $\psi_{|N^i} : N^i \rightarrow \dga^i$ such that the diagram
\[
\xymatrix{
N^i \ar^-{d}[r] \ar_-{\psi}[d] & \clalg^{i+1}(\Lambda V^{<i}) \ar^-{\psi}[d] \\
\dga^i \ar^-{d}[r] & \dga^{i+1}
}
\]
commutes (such a $\psi$ exists, but in general it is not unique). 

By construction the $(\Lambda V, \psi)$ we get is a minimal model of $\dga$. Next we give a more explicit description of the spaces $C^{i+1}$ and $N^i$ (in the relevant range), and, using an appropriate choice of $\psi$ and the assumptions on $c$, prove that $\Lambda V$ is $s$-formal (with the chosen complements $N^i$).

\begin{itemize}
\item
We must in any case take $V^i = 0$ for $0 \leq i < n$.

\item
For $i \leq 2n-2$ we must also have $N^i = 0$, because the degree $i+1$
part of $\Lambda V^{<i}$ is trivial.

\item
For $n \leq i \leq 2n-1$ we take $C^i = H^i$, and on $C^i$ we define $\psi$ by
$\alpha$.

\item For $2n-1 \leq i \leq 3n-3$, the degree $i+1$ part of $\Lambda V^{<i}$
is simply $(\gsym^2 C^{<i})^{i+1} \cong (\gsym^2 H^{<i})^{i+1}$, 
and the kernel of $(\gsym^2 C^{<i})^{i+1} \to H^{i+1}$ corresponds to
$\kc^{i+1}$. We can take $N^i = \kc^{i+1}$, and define $\psi$ on $N^i$
by $\gamma$.

\item Similarly, for $2n \leq i \leq 3n-2$ we can take $C^i$ to be a direct
complement in $H^i$ to the image of $(\gsym^2 H^{\leq i-2})^i \to H^i$,
and define $\psi$ on $C^i$ by $\alpha$.
\end{itemize}

To complete the description of the part of the minimal model that is relevant
for $(3n{-}2)$-formality, it remains only to describe $N^{3n-2}$.
Now the degree $3n-1$ part of $\Lambda V^{\leq 3n-3}$ is a direct sum of
of $(\gsym^2 C^{\leq 3n-3})^{3n-1}$ and
$C^n \otimes N^{2n-1} \cong H^n \otimes \kc^{2n}$. 
The closed subspace is the direct sum of $(\gsym^2 C^{\leq 3n-3})^{3n-1}$ and
the closed subspace $\calk \leq C^n \otimes N^{2n-1}$, which corresponds to
$\kersym{H^n \otimes \kc^{2n}}$.
On the first summand, the map to $H^{3n-1}$ is simply the cup product, so
$N^{3n-3}$ should include a summand $\wt N^{3n-2}$ isomorphic to $\kc^{3n-1}$.
On the other summand $\calk$, it is induced by the restriction of
$\alpha\gamma : H^n \otimes \kc^{2n} \to \dga^{3n-1}$, \ie it corresponds to
the uniform triple product
$\trip_c : \kersym{H^n \otimes \kc^{2n}} \to H^{3n-1}$
defined by the cochain choice $c = (\alpha, \gamma)$.

Recall that $(\alpha, \gamma)$ has been chosen such that
$\trip_c = 0$ (\ie $\alpha\gamma$ takes exact values on $\calk$).
Thus we should take $N^{3n-2} = \wt N^{3n-2} \oplus \calk$, 
defining $\psi$ by $\gamma$ on $\wt N^{3n-2}$, and by some choice of predifferential of $\alpha\gamma$ on $\calk$.

Using that we also have $\nmt_c = 0$, we now want to prove that
it is possible to choose $\psi$ on $\calk$ such that $\Lambda V$ is a $(3n{-}2)$-formal minimal model.
Let $\wh \clalg^* \subseteq N^{\leq 3n-2} \Lambda V^{\leq 3n-2}$ be the closed
subalgebra.
 Every $w \in \wh\clalg^r$ with $r \leq 3n{-}2$
 is exact in $\Lambda V$ by construction, so by Remark \ref{rmk:reduce} it
remains only to consider $\wh\clalg^m$.

To decompose $\wh\clalg^m$, set $\wt N^i = N^i$ for $i \leq 3n-3$
(so that $\wt N^i \cong \kc^{i+1}$ for each $i \leq 3n-2$),
and $\wt V = \bigoplus_{i= 0}^{3n-2} C^i \oplus \wt N^i \leq V^{\leq 3n-2}$
(so $V^{\leq 3n-2} = \wt V \oplus \calk$).
Let $\wt\clalg^* \subset (\wt N \Lambda \wt V)^*$ be the closed subalgebra.

\begin{lem}
For $m = 5n-2$, every closed element of $\wt\clalg^m$ is exact in $N \Lambda V$.
\end{lem}

\begin{proof}
$\Lambda \wt V$ can be decomposed as a direct sum of the graded subspaces
\[ \subsp^*(a,b) = \gsym^a C^* \otimes \gsym^b \wt N^* . \]
Because $m \leq 5n-2$, the degree $m$ part of $\wt N \Lambda \wt V$
is a direct sum of $\subsp^m(1, 1)$, $\subsp^m(2,1)$, $\subsp^m(0,2)$ and $\subsp^m(1,2)$.
The differential maps $\subsp^i(a,b) \to \subsp^{i+1}(a+2, b-1)$, so
$\wt\clalg^m$ is a direct sum of the closed
subspaces $Z^m(a,b) \leq \subsp^m(a,b)$.
It thus remains to check for each of these $Z^m(a,b)$ that
$Z^m(a,b) \to H^m$ is trivial.

On $\subsp^*(1,1) \cong H^* \otimes \kc^*$, the differential corresponds to the
symmetrisation map $H^* \otimes \kc^* \to \gsym^3 H^*$, so
$Z^*(1,1) \cong \tripsp$. $Z^*(1,1) \to H^*$ is $\trip_c$, which is trivial
by the choice we made for $c$.

On $\subsp^*(2,1) \cong \gsym^2 H^* \otimes \kc^*$, the differential
corresponds to the 
symmetrisation $\gsym^2 H^* \otimes \kc^* \to \gsym^4 H^*$.
By the Bianchi identity, its kernel $Z^*(2,1)$ equals the image of
the natural map from $H^* \otimes \tripsp$.
The map $Z^*(2,1) \to H^*$ corresponds to $\alpha \trip_c$, so again vanishes
because we ensured $\trip_c = 0$.

On $\subsp^*(0,2) \cong \agsym^2 \kc^*$ the differential is the composition $\agsym^2 \kc^* \rightarrow \kc^* \otimes \kc^* \rightarrow \gsym^2 H^* \otimes \kc^*$ of two injections, hence $Z^*(0,2) \cong 0$. 

On $\subsp^*(1,2) \cong H^* \otimes \agsym^2 \kc^*$ the differential
corresponds to $\rmap$ from \eqref{eq:rmap}, so the kernel is 
$Z^*(1,2) \cong \ndom$. The induced map $Z^*(1,2) \to H^*$ is precisely
$\nmt_c$.
\end{proof}

Now let $D^{2n} \leq \gsym^2 C^n \oplus C^{2n} \leq (\Lambda \wt V)^{2n}$
be a direct complement to $\kc^{2n}$ (so that $D^{2n} \cong H^{2n}$).

\begin{lem}
For $m = 5n-2$, the closed subspace 
$\wh\clalg^m \subseteq (N^{\leq 3n-2} \Lambda V^{\leq 3n-2})^m$ is contained in
\begin{equation}
\label{eq:decomp}
\calk \otimes D^{2n} \; \oplus \;
d\left(\calk \otimes N^{2n-1}\right) \; \oplus \; (\wt N \Lambda \wt V)^m.
\end{equation}
\end{lem}

\begin{proof}
For degree reasons we have
$(N^{\leq 3n-2} \Lambda V^{\leq 3n-2})^m =
(\wt N \Lambda \wt V)^m \; \oplus \; \calk \otimes (\Lambda \wt V)^{2n}$.

The differential maps $\calk \otimes N^{2n}$ to
$\calk \otimes C^n \otimes C^{n+1} \; \oplus \;
C^n \otimes N^{2n+1} \otimes N^{2n}$.
The composition with the projection to the first summand is injective because
the differential $N^{2n} \to C^n \otimes C^{n+1}$ is.
Therefore $\wh\clalg^m$ must be contained in
$(\wt N \Lambda \wt V)^m \oplus \calk \otimes (C^{2n} \oplus \gsym^2 C^n) = 
(\wt N \Lambda \wt V)^m \oplus \calk \otimes (D^{2n} \oplus \kc^{2n})$.
It thus remains only to note that for any element
$ke \in \calk \otimes \kc^{2n}$, there is a $y \in N^{2n-1}$ with $dy = e$,
so $ke - d(ky) = y(dk)  \in (\wt N \Lambda \wt V)^m$.
\end{proof}

Finally we deduce that $\psi : \calk \to \dga^{m-2n}$ can be chosen so that that
every element in the closed subspace
$\wh\clalg^m \subseteq (N^{\leq 3n-2} \Lambda V^{\leq 3n-2})^m$ is exact in
$\Lambda V$.
Let $W \subseteq \wh\clalg^m$ be a direct complement to $\wt\clalg^m
\oplus d\left(\calk \otimes N^{2n-1}\right)$.
We already know that $\psi_* : \wt\clalg^m
\oplus d\left(\calk \otimes N^{2n-1}\right) \to H^m$ is trivial.
On the other hand, the projection to the $\calk \otimes D^{2n}$ summand in
\eqref{eq:decomp} is injective on $W$. Therefore by Poincar\'e duality we
can adjust the values of $\psi_* : W \to H^*$ arbitrarily by
changing the choice of
$\psi$ on the $\calk$ summand by adding on some map $\calk \to \clalg^{m-2n}$.
In particular we can choose it such that $\psi_*$ vanishes on all
of $\wh\clalg^m$.
This completes the proof of Theorem \ref{thm:formality}.
\end{proof}

\begin{rmk}
\label{rmk:gen_bmp}
If one extended the dimension range to consider closed \tkc{} manifolds of
dimension $m \leq 6n-4$, the only part of the proof that would fail is that
the degree $m$ part of $\wt N \Lambda \wt V$ could also have a contribution
from $\subsp^m(3,1)$, so $\wt\clalg^m$ could have a non-trivial $Z^m(3,1)$
summand. We expect that to prove formality in this case would
require a further generalisation of the \bmp.

For any $r$, we could consider the kernel $E_r^*$ of the $r$-fold product
$\gsym^r H^* \to H^*$. Given a right inverse $\alpha : H^* \to \clalg^*$,
one could then pick a predifferential $\gamma_r : E_r^* \to \dga^{*-1}$
of the restriction to $\kc_r^*$ of $\alpha^r : \gsym^r H^* \to \dga^*$.
The restriction of the degree $-1$ map
\[ \gamma_r \alpha^s : E_r^* \otimes E_s^* \to \dga^{*-1} \]
to the kernel $\kersym{E_r^* \otimes E_s^*}$ of full graded
symmetrisation $E_r^* \otimes E_s^* \to \gsym^{r+s}H^*$ takes closed values.
The induced map $\princ_{r,s} : \kersym{E_r^* \otimes E_s^*} \to H^{*-1}$ is
clearly independent of the choice of $\gamma$, and seems likely to be
independent of the choice of $\alpha$ too.

$\princ_{2,2}$ amounts to the Bianchi-Massey tensor. Depending on the product
structure on $H^*$, other $\princ_{r,s}$ may be determined by the
Bianchi-Massey tensor too, but it seems it does not always need to be.
Vanishing of $\princ_{3,2}$ seems relevant for formality a \tkc{} manifold of
dimension $5n-2 < m \leq 6n-4$.

Closed \tkc{} manifolds of dimension $> 6n-4$ can have non-trivial 5-fold
Massey products, so any criterion for formality of such manifolds would
need to include some invariant that controls those products.
\end{rmk}

\subsection{Using \texorpdfstring{$\nmt$}{P} to distinguish and classify manifolds}

For an isomorphism $F : H^*(X) \to H^*(Y)$ to be realised by a rational
homotopy equivalence, it is clearly necessary that there are some
choices $b$ and $c$ on $X$ and $Y$ such that $F$ intertwines
the uniform triple products $\trip_b$ and $\trip_c$, and
the pentagonal Massey tensors $\nmt_b$ and $\nmt_c$.
Supposing that we can make the choices so that $\trip_b$ and $\trip_c$
are intertwined, let us now consider how to measure the failure of the
pentagonal Massey tensors to agree.

Let $\Delta \subseteq \hom(\ndom^*, H^{*-2})$ be the space of degree $-2$ maps
$\ndom^* \to H^{*-2}$ generated by functions
of the form $\trip_c\delta \circ (\idf j)$, for
$\delta : \kc^* \to H^{*-1}$ such that the restriction
of $\delta\idf$ to $\tripsp$ vanishes.
Then Lemma \ref{lem:same_trip} (together
with Theorem \ref{thm:choices}) means that whenever two choices $c'$ and
$c$ have $\trip_{c'} = \trip_c$, then $\nmt_{c'} - \nmt_c \in \Delta$.
In consequence, if there are some cochain choices $b$ on $X$ and $c$ on $Y$
such that $\trip_b$ and $\trip_c$ are intertwined by $F$, then
\[ F^\# \nmt_c - \nmt_b \in \hom(\ndom^*, H^{*-2})/\Delta \]
takes the same value for all such pairs of choices, and vanishes if and only
if there is some pair for which $F^\# \nmt_c = \nmt_b$.

In the case of closed manifolds, the existence of choices such that
$\trip_b$ and $\trip_c$ are intertwined by $F$ is equivalent to the more
convenient condition that the \bmp s (which do not depend on choices at all)
are intertwined by $F$ \cite[Lemma 2.8(ii)]{bmp}.

For closed \tkc{} manifolds of dimension up to $5n-3$, Massey products of order
4 or higher (and analogous data like $\nmt$) are irrelevant for degree reasons.
In this case, \cite[Theorem~1.2]{bmp} shows that $F$ is realised by
a rational homotopy equivalence if and only it intertwines the \bmp s.

For closed \tkc{} manifolds of dimension up to $6n-4$, Massey products of
order 5 or greater are irrelevant. While we expect that $\nmt$ captures
all the ``fourfold product like'' information, Remark \ref{rmk:gen_bmp}
indicates that some data may need to be added to the Bianchi-Massey 
and pentagonal Massey tensors to determine the rational homotopy type if the
dimension is greater than $5n-2$.
However, the following statement is reasonable in view of the formality
criterion of Theorem \ref{thm:main2}.

\begin{conj}
\label{conj:type}
Let $F : H^*(X) \to H^*(Y)$ be an isomorphism of the cohomology algebras
of closed \tkc{} $m$-manifolds, for $n > 1$ and $m \leq 5n-2$.
Then $F$ is realised by a rational homotopy equivalence if and only if
\begin{itemize}
\item $F$ intertwines the \bmp s, and
\item for any choices $b$ and $c$ such that $F$ intertwines $\trip_b$
and $\trip_c$, we have $F^\# \nmt_c - \nmt_b \in \Delta$. 
\end{itemize}
\end{conj}

In view of Remark \ref{rmk:bmt_zero}, Conjecture \ref{conj:special} is a
special case of Conjecture \ref{conj:type}.

\pagebreak[2]

\bibliographystyle{amsinitial}
\bibliography{g2geom}

\end{document}